\documentclass[pdflatex,sn-mathphys-num]{sn-jnl}
\usepackage{graphicx}%
\usepackage{multirow}%
\usepackage{amsmath,amssymb,amsfonts}%
\usepackage{amsthm}%
\usepackage{mathrsfs}%
\usepackage[title]{appendix}%
\usepackage{xcolor}%
\usepackage{textcomp}%
\usepackage{manyfoot}%
\usepackage{booktabs}%
\usepackage{algorithm}%
\usepackage{algorithmicx}%
\usepackage{algpseudocode}%
\usepackage{listings}%

\theoremstyle{thmstyleone}%
\newtheorem{theorem}{Theorem}
\theoremstyle{thmstyletwo}%
\newtheorem{remark}{Remark}%

\theoremstyle{thmstylethree}%
\newtheorem{definition}{Definition}%

\usepackage{amssymb}
\usepackage{amsthm}
\usepackage{tikz}
\usepackage{makeidx}         
\usepackage{graphicx}        
\usepackage{multicol}        
\usepackage[bottom]{footmisc}

\usepackage{newtxtext}       %
\usepackage{newtxmath}       
\usepackage{multirow}
\usepackage{diagbox}

\usepackage{xcolor}

\usepackage{cleveref}
\usepackage{booktabs}
\usepackage{geometry}

\definecolor{darkgreen}{rgb}{0.0,0.4,0.0}
\definecolor{darkred}{rgb}{0.6,0.0,0.0}
\definecolor{darkblue}{rgb}{0.0,0.0,0.5}
\definecolor{gray}{rgb}{0.5,0.5,0.5}
\definecolor{cyan}{rgb}{0.0,1.0,1.0}
\definecolor{darkcyan}{rgb}{0.0,0.5,0.5}
\definecolor{darkorange}{rgb}{0.8,0.4,0.0}
\definecolor{darkmargenta}{rgb}{0.5,0.0,0.5}
\definecolor{black}{rgb}{0.0,0.0,0.0}

\hypersetup{citecolor=blue, linkcolor=red, anchorcolor=darkblue,
	filecolor=darkblue, urlcolor=darkblue}
\hypersetup{pdfauthor={Jakob Seierstad Stokke},pdftitle={Stokke}}

\usepackage{float} 
\usepackage{bm}

\makeindex             

\numberwithin{equation}{section}

\def \a  {\alpha}

\def \d  {\delta}

\def \e  {\varepsilon}

\def \om {\omega}
\def \Om {\Omega}

\def \P {\mathbb{P}}

\def \calA {\mathcal{A}}

\def \t  {\tau}

\def \del {\nabla}
\def \div {\nabla\cdot}

\def \p  {\partial}
\def \R  {\mathbb{R}}
\def \N  {\mathbb{N}}

\def \bmu {\bm{u}}
\def \bmv {\bm{v}}

\def \bmF {\bm{f}}

\def \glp {t_{n,\mu}}

\def \bmPh {\bm{\Phi}}
\def \bmPth {\bm{\Phi}_{\tau,h}}

\def \bmE {\bm{E}}

\def \bmw {\bm{w}}
\def \bmR {\bm{R}}

\def \bmEth {\bm{E}_{\tau,h}}

\usepackage[section]{placeins}

\newtheorem{lemma}[theorem]{Lemma}

\theoremstyle{definition}
\newtheorem{problem}{Problem}

\newcommand{\vertiii}[1]{{\left\vert\kern-0.17ex\left\vert\kern-0.17ex\left\vert #1 
		\right\vert\kern-0.17ex\right\vert\kern-0.17ex\right\vert}}

\newcommand{\norm}[1]{{\left\vert\kern-0.25ex\left\vert\kern-0.25ex\left\vert #1 
		\right\vert\kern-0.25ex\right\vert\kern-0.25ex\right\vert}}

\makeatletter
\newsavebox{\@brx}
\newcommand{\llangle}[1][]{\savebox{\@brx}{\(\m@th{#1\langle}\)}%
	\mathopen{\copy\@brx\kern-0.5\wd\@brx\usebox{\@brx}}}
\newcommand{\rrangle}[1][]{\savebox{\@brx}{\(\m@th{#1\rangle}\)}%
	\mathclose{\copy\@brx\kern-0.5\wd\@brx\usebox{\@brx}}}
\makeatother

\usepackage{mdframed}
\usepackage{array}

\newcolumntype{C}[1]{>{\centering\arraybackslash}p{#1}}
\usepackage{nomencl}
\makenomenclature

\usepackage{etoolbox}
\renewcommand\nomgroup[1]{%
	\item[\bfseries
	\ifstrequal{#1}{A}{Physical quantities}{%
		\ifstrequal{#1}{N}{Space and time}{%
			\ifstrequal{#1}{B}{Function spaces}{%
				\ifstrequal{#1}{C}{Discretization}{%
					\ifstrequal{#1}{K}{Abbreviations}{}}}}}%
	]}

\begin{document}
	
	\title[Convergence of a continuous Galerkin method for the  Biot--Allard poroelasticity system]{Convergence of a continuous Galerkin method for the  Biot--Allard poroelasticity system}
	
	
	\author*[1]{\fnm{Jakob Seierstad} \sur{Stokke}}\email{jakob.stokke@uib.no}
	
	\author[2]{\fnm{Markus} \sur{Bause}}\email{bause@hsu-hh.de}

	\author[1]{\fnm{Florin Adrian} \sur{Radu}}\email{florin.radu@uib.no}

	\affil*[1]{\orgdiv{Center for Modeling of Coupled Subsurface Dynamics}, \orgname{Department of mathematics, University of Bergen}, \orgaddress{\street{ Allegaten 41}, \city{Bergen}, \country{Norway}}}
	

	\affil[2]{\orgdiv{Faculty of Mechanical and Civil Engineering}, \orgname{Helmut Schmidt University}, \orgaddress{\street{Holstenhofweg 85},  \state{Hamburg}, \country{Germany}}}

	
	\abstract{We study a space-time finite element method for a system of poromechanics with memory effects that are modeled by a convolution integral. In the literature, the system is referred to as the Biot--Allard model. We recast the model as a first-order system in time, where the memory effects are transformed into an auxiliary differential equation. This allows for a computationally efficient numerical scheme. The system is discretized by continuous Galerkin methods in time and equal-order finite element methods in space. An optimal order error estimate is proved for the norm of the first-order energy of the unknowns of the system. The estimate is confirmed by numerical experiments. }

	
	
	
	\maketitle

	\section{Introduction}
	\label{sec:intro}
	
	We consider a poroelastic system, which models the coupling of flow and deformation in a fully saturated porous medium and includes memory effects, referred to as the Biot--Allard model. The model is derived through asymptotic homogenization in the space and time variables by considering linear fluid–structure equations on the pore scale, see \cite[Eqs (3)-(7)]{mikelic2012}, resulting in the problem; Find the solid phase displacement $(\bmu)$ and the pressure $(p)$ such that	
	\begin{subequations}\label{eqs:BiotAllardconv}
		\begin{align}
			\rho\p_{t}^2\bmu-\div\bm{C}\e(\bmu)+\a\del p+\p_t\left(\calA\ast (\rho_f\bm{f}-\del p-\rho_f\p_t^2\bmu)\right)&=\rho\bm{f}\label{eq:biotallardmechanincs} \quad\mbox{ in }\Om\times(0,T], \\	
			c_0\p_tp+\div (\alpha\p_t\bmu)+\div \left(\calA\ast(\bmF-\frac{1}{\rho_f}\del p-\p_t^2\bmu)\right)&=0\label{eq:biotallardflow} \quad\quad\mbox{ in }\Om\times(0,T],
		\end{align}
		together with initial and boundary conditions
		\begin{align}\label{eq:conditions}
			\bmu(0)=\bmu_0,\quad\p_t\bmu(0)=\bmu_1, \quad p(0)&=p_0, \quad\mbox{ in }\Om,\\
			\bmu=\bm{0}, \quad p&=0,\quad\,\,\,\mbox{ on }\p\Om\times(0,T],
		\end{align}	
		where $\ast$ denotes the time convolution operator, i.e. $f(t)\ast g(t)=\int_0^{t}f(t-s)g(s)ds$, $\Om\subset\R^{d}, d\in\{2,3\}$ is a bounded domain with boundary $\p\Om$ and $T$ is the final time. Here, $\rho$ is the mass density, i.e. $\rho=\rho_f\varphi+\rho_s(1-\varphi)$ with $\rho_f$ being the fluid density, $\rho_s$ being the solid density and $\varphi$ the porosity. Further, $c_0$ is the specific storage capacity, $\a$ is the Biot coefficient, $\e(\bmu):=\frac12\left(\del\bmu+\del\bmu^{T}\right)$ is the linearized strain tensor and $\bm{C}$ is Gassmann's fourth order tensor. The dynamic permeability tensor is denoted by $\calA$, accounting for the delay memory effects due to the interaction between the fluid and the pore walls. For a general definition of the dynamic permeability, see \cite{johnson1987}.  
	\end{subequations}
	
	There is a wide variety of applications of the poroelastic theory today, ranging from bio-mechanics and bio-medicine to subsurface engineering. For example, in bone remodeling, the activation of bone cells is closely related to memory effects due to fluid flow along cell bodies \cite{robling2006}. Another example is induced seismicity due to fluid injection. In this case, the delay effects in the Biot--Allard model can be viewed as a consequence of drag forces in the pores, which in large-scale experiments can lead to different onset times for the seismic event. It is worth noting that small-scale experimental results in \cite{batzle2001} indicate that memory effects due to drag forces are relevant for seismic wave propagation in rocks. In the case of a porous medium saturated with air, i.e., noise protection applications, the memory terms become important \cite{allard2009}. Therefore, the ability to simulate poroelastic systems is an important task from the perspective of physical realism.

	The poroelastic theory, often referred to as Biot theory due to the development of the field by M.A. Biot in a series of papers \cite{biot1941,biot1955,biot1956,biot1962}, has received a lot of attention because of the many above-mentioned relevant applications.  The well-posedness of Biot systems, including the existence, uniqueness, and regularity of solutions, has been considered by many \cite{auriault1977,showalter1974,zenisek1984,showalter2000,both2019}. Numerical analysis of the quasi-static Biot model, which can be interpreted as the singular limit of \eqref{eqs:BiotAllardconv} \cite{mikelic2012}, has been studied a lot over the last decades. One of the earlier stability and convergence results considered inf-sup stable finite elements with a backward Euler discretization of a two-field formulation of the Biot system \cite{murad1994}. However, optimal \textit{a priori} $L^{2}$ estimates have only recently been derived in \cite{wheeler2022}. For the three-field formulation, a stable conforming method was proposed in \cite{rodrigo2018}. Due to a large parameter variation, up to several orders of magnitude depending on the application, many works have been on designing robust preconditioners with respect to all parameters \cite{lee2017,hong2018}. More recently, using space-time finite elements, the convergence of the dynamic Biot system without memory effects was shown \cite{bause2024,kraus2024}.
	
	Computational efficiency has been a major aspect of solving the Biot system. In this respect, splitting schemes, i.e., schemes that sequentially solve the different subproblems, have become very popular. Two commonly used iterative splitting schemes are the fixed-stress and the undrained split. The idea behind them is to add a stabilization term to either the flow or the mechanics subproblem \cite{kim2011}. The convergence of these schemes for the quasi-static model was studied in \cite{mikelic2013,both2017,bause2017b,storvik2019}. The splitting schemes can also be used as preconditioners \cite{white2011}. For the dynamic Biot model, the undrained split has been proven to converge \cite{bause2021}, using the gradient flow technique \cite{both2019}. The iterative methods have also been expanded and studied for a soft poromechanical model \cite{both2022}. Furthermore, the fixed-stress approach has been extended to a simplified Biot--Allard model with convolution integrals \cite{stokke2024b}.

	In this paper, we seek to solve and rigorously show the convergence of a variational space-time finite element method for the Biot-Allard system. The main challenge is how to deal with the convolution integral. Our solution is to use a result from \cite{stokke2024} where they rewrite the convolution integral into an auxiliary differential equation (ADE) by using a series approximation of the dynamic permeability \cite{yvonne2014}. The approximation of the dynamic permeability is a finite sum of decaying exponential functions in time, which is a natural choice under a fading-memory hypothesis. Each exponential term introduces an auxiliary variable, yielding $N$ additional ADEs.  Note the difference in sign convention in the Fourier transform between these references. The main advantage now is that one avoids expensive and memory-demanding integration rules.  
	For simplicity in the analysis, we consider the case when $N=1$ \cite[Eq.8]{stokke2024}, which is equivalent to assuming $\mathcal{A}$ is represented by a single exponential decay function. In Appendix~\ref{appB}, numerical examples comparing the JKD dynamic permeability model \cite{johnson1987} to the approximation are included. The examples show that the $N=1$ approximation can be sufficiently accurate, while also indicating that a larger $N$ may be needed to represent the dynamic permeability more accurately over a broader frequency range. The model for $N=1$ is: Find the solid phase displacement $(\bmu)$, the pressure $(p)$ and the memory variable $(\Psi)$ such that
	\begin{subequations}\label{eqs:biotallardADE}
		\begin{align}
			\rho\p_{t}^2\bmu-\div\bm{C}\e(\bmu)+\a\del p +\rho_f\p_t \Psi&=\rho\bmF,\quad&&\mbox{ in }\Om\times(0,T],\\
			c_0\p_tp+\div (\alpha\p_t\bmu)+\div \Psi&=0,\quad &&\mbox{ in }\Om\times(0,T],\\
			c_1\rho_f\p_t\Psi +\rho_f\Psi+\frac{d_1\eta_k}{F}\del p+\frac{d_1\eta_k}{F}\rho_f\p_t^{2}\bmu&=\frac{d_1\eta_k}{F}\rho_f\bmF,\quad &&\mbox{ in }\Om\times(0,T],
		\end{align}
		with initial and boundary conditions
		\begin{align}
			\bmu(0)=\bmu_0,\quad\p_t\bmu(0)=\bmu_1, \quad p(0)=p_0,\quad\Psi(0)=\Psi_0, \quad&\mbox{ in }\Om,
			\\
			\bmu=\bm{0}, \quad p=0,\quad \Psi=0,\quad&\mbox{ on }\p\Om\times(0,T].
		\end{align}	
	\end{subequations}
	Here, $\eta_k$ is the kinetic viscosity, $F$ the formation factor, and $c_1,\, d_1>0$ are constants that can be estimated through a numerical procedure detailed in \cite{yvonne2014}. Homogeneous Dirichlet boundary conditions are employed solely to simplify the error analysis. We note that in the case of wave propagation due to a seismic event, absorbing boundary conditions should be used (see, e.g., \cite{clayton1977}). The well-posedness of \eqref{eqs:biotallardADE} is studied in \cite{stokke2024}.  For the system \eqref{eqs:biotallardADE}, we will consider space-time finite element methods based on a variational formulation in both time and space. Before discretizing in time, we will rewrite the system as a first-order system in time by defining $\bmv:=\p_t\bmu$. In time, we will use continuous Galerkin methods, as they ensure energy conservation for second-order wave equations \cite{bause2020}. For the spatial discretization, we will use an equal-order approximation. Variational discretizations in space and time offer appreciable advantages for the discretization of systems of coupled partial differential equations. Higher-order methods are directly constructed by increasing the piecewise polynomial order. Coupling terms, mixed space-time derivatives, and nonlinearities are discretized naturally; cf.\ Subsec.~\ref{Subsec:CGD}. The uniform variational approach simplifies stability and error analyses by the applicability of tools from functional analysis \cite{bause2020,makridakis2005,steinbach_generalized_2022}. Further, goal-oriented \textit{a posteriori} error control based on the dual weighted residual approach relies on variational space–time formulations and becomes feasible \cite{bangerth_adaptive_2010}. Concepts of adaptively changing the polynomial degree as well as post-processing techniques \cite{ern_discontinuous_2016} can be used. In particular, strong relations between continuous Galerkin schemes, collocation, and Runge--Kutta methods have been observed. In \cite{akrivis_optimal_2009,akrivis_galerkin_2011} they are studied thoroughly. The novelty of the resulting numerical scheme is that we have replaced the convolution integral with an ADE that incorporates memory effects. The analysis of the scheme uses an energy quantity of the system \eqref{eqs:biotallardADE}. This is necessary to control the errors arising from the coupling terms and is further motivated by the stability estimate \eqref{eq:errorstabilityestimate} of the continuous solution in the same combined norm. Therefore, separating the errors in the unknowns is not feasible in our analysis. The energy analysis does not depend on the inf-sup stability condition, which allows us to use an equal-order approximation.  
	
	We prove error estimates in a natural energy norm for the system \eqref{eqs:biotallardADE}. Precisely, we show that the approximations $(\bmu_{\t,h},\bmv_{\t,h},p_{\t,h},\Psi_{\t,h})$ of temporal order $k\geq 1$ and spatial order $r\geq 1$ satisfy 
	\begin{equation}\label{eq:introerror}
		\begin{aligned}
			\max_{t\in [0,T]}\left\{\|\del(\bmu(t)-\bmu_{\t,h}(t))\|+\left\|\bm{N}^{\frac12}\begin{pmatrix}
				\bmv(t)-\bmv_{\t,h}(t)\\\Psi(t)-\Psi_{\t,h}(t)
			\end{pmatrix}\right\|+c_0\|p(t)-p_{\t,h}(t)\|\right\}\\\leq c\left(\t^{k+1}+h^{r}\right),
		\end{aligned}
	\end{equation}
	where $\bm{N}$ is defined in \eqref{eq:Nmatrix} and $\| \cdot \|$ stands for the $L^2$-norm.

	To show the convergence of the scheme, we build on the work of \cite{bause2024}, utilizing key ideas from \cite{makridakis1998,makridakis2005}.
	The analysis in the next sections follows several steps, which we describe briefly here. In \Cref{sec:error}, we have included a schematic of the proof; see \Cref{fig:schematic}. 
	\begin{itemize}
		\item We derive variational equations for the discrete errors (\Cref{lemma:varerror}) and identify which terms on the right-hand side need to be estimated when they are summed together (\Cref{lemma:estimation}).
		\item We then derive stability estimates for the discrete errors (\Cref{theorem:stability}), where the key idea is choosing the correct test functions to gain control of the errors.
		\item Next, we want to estimate additional terms of the right-hand side in the stability estimate (\Cref{lemma:newname}). The central idea for this proof is similar to \cite{makridakis1998,makridakis2005}, where we utilize a representation of the discrete errors in terms of Gauss nodes at each time interval $I_n$. We can then control the norm of the error representation by \cite[Lem. 2.1]{makridakis2005}.
		\item Then, we relate the jump in the discrete errors at the endpoints of two adjacent subintervals in time (\Cref{lemma:estimation2}).
		\item By estimating additional terms resulting from partial integration in time when deriving the stability estimate, we can combine all the results. The proof is then concluded by a discrete Gronwall argument.
	\end{itemize}
	\begin{figure}[ht]
		\centering
		\begin{tikzpicture}
			\node[scale=0.6] {\includegraphics[width=\textwidth]{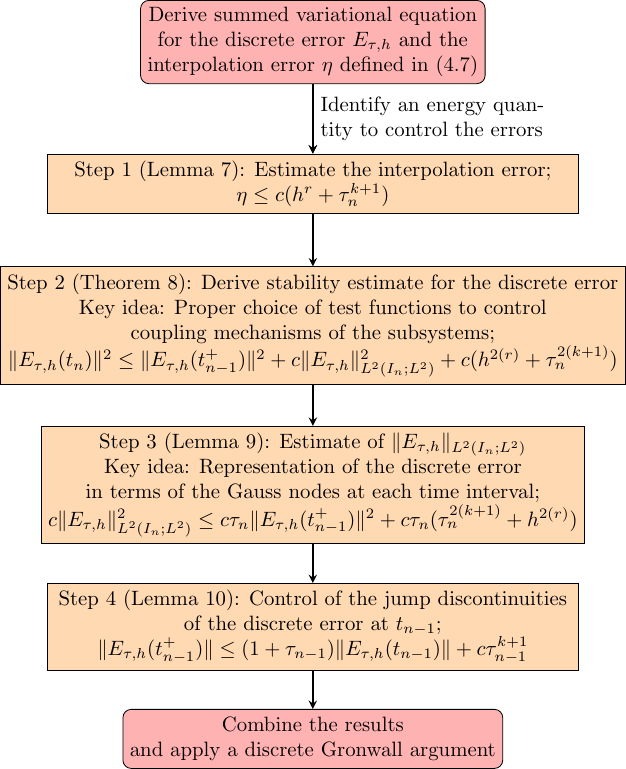}};
		\end{tikzpicture}
		\caption{Schematic of proof strategy. Error estimation by bounding the discrete error $E_{\tau,h}$ and the interpolation error $\eta$, each associated with all unknowns of the system. }
		\label{fig:schematic}
	\end{figure}

	We also conjecture and show numerically that for the non-equal order inf-sup stable elements, where $p$ is of order $r\geq 1$ and $\bmu,\bmv$ and $\Psi$ are of order $r+1$ in space while still being order $k\geq 1$ in time, that the approximations $(\bmu_{\t,h},\bmv_{\t,h},p_{\t,h},\Psi_{\t,h})$ satisfy 
	\begin{equation}\label{eq:conjecture}
		\begin{aligned}
			\max_{t\in [0,T]}\left\{\|\del(\bmu(t)-\bmu_{\t,h}(t))\|+\left\|\bm{N}^{\frac12}\begin{pmatrix}
				\bmv(t)-\bmv_{\t,h}(t)\\\Psi(t)-\Psi_{\t,h}(t)
			\end{pmatrix}\right\|+c_0\|p(t)-p_{\t,h}(t)\|\right\}\\\leq c\left(\t^{k+1}+h^{r+1}\right).
		\end{aligned}
	\end{equation}

	The paper is organized in the following manner. In \Cref{sec:prelim}, we introduce the notation and some auxiliary results. In \Cref{sec:disc} we present the numerical scheme and in \Cref{sec:error} we prove its convergence \eqref{eq:introerror}. Finally, numerical experiments are presented in \Cref{sec:numerics}. Additionally, Appendix~\ref{appB} includes a numerical convergence test comparing the ADE solution with the convolution integral. 
	\section{Preliminaries}\label{sec:prelim}
	In this paper, we use standard functional analysis notation. By $L^{2}(\Om)$, we denote the space of Lebesgue measurable and square-integrable functions defined on $\Omega$. We let $H^{m}(\Om)$ be the Sobolev space of $L^{2}(\Om)$ functions with weak derivatives up to order $m$ in $L^{2}(\Om)$ and $H^{1}_{0}(\Om):=\left\{f\in H^{1}(\Om)\,\, |f=0 \mbox{ on }\p\Om\right\}$. The dual space of $H^{1}_{0}(\Om)$ is denoted $H^{-1}(\Om)$. Throughout the paper we skip the domain $\Om$ to shorten notation, i.e. $L^{2}:=L^{2}(\Om), H^{m}:=H^{m}(\Om), H^{1}_{0}:=H^{1}_{0}(\Om)$ and $H^{-1}:=H^{-1}(\Om)$. In the case of vector-valued functions, the spaces are written in bold, e.g., $\bm{L}^{2}=L^{2}(\Om)^{d}$. The associated norms of the Sobolev spaces are denoted by $\|\cdot\|=\|\cdot\|_{L^{2}}, \|\cdot\|_{m}=\|\cdot\|_{H^{m}}$ for $m\in\N_0$. We define an $L^{2}$ inner product on the product space $\bm{L}^{2}\times\bm{L}^{2}$ by
	\begin{align*}
		\llangle\bm{v},\bm{w}\rrangle=\langle \bmv^{1},\bmw^{1}\rangle_{\bm{L}^{2}}+\langle \bmv^{2},\bmw^{2}\rangle_{\bm{L}^{2}},\quad\mbox{ for } \bmv,\bmw\in\bm{L}^{2}\times\bm{L}^{2}\mbox{ where }\bmv=(\bmv^{1},\bmv^{2}),\,\bmw=(\bmw^{1},\bmw^{2}).
	\end{align*}  Given a Banach space $B$, we define the Bochner spaces of $B$-valued functions
	\begin{align*}
		C([0,T];B)&:= \left\{f:[0,T]\to B\,\, | \,\,f\mbox{ is continuous}\right\},
		\\
		L^{2}((0,T);B)&:=\left\{f:(0,T)\to B \,\,\Big|\,\, \int_{0}^{T}\|f(t)\|^{2}_{B}\,dt<\infty\right\},
	\end{align*} equipped with their natural norms (cf. \cite{ern2004}). When considering a sub-interval $J$ of $[0,T]$, i.e. $J\subset[0,T]$, we will use the following notation for the Bochner spaces $L^{2}(J;B)$ and $C(J;B)$. We note that the constants $c$, which appear in the paper, are generic and do not depend on the size of the spatial and time meshes. It can depend on the norms of the solution, the regularity of the spatial mesh, the polynomial degrees used for the space-time discretization, and the data.
	
	\subsection{Finite element spaces}
	For the discretization in time, we split the time interval $I=(0, T]$ into $N$ subintervals $I_{n}=(t_{n-1},t_{n}]$, with  $n\in\{1,...,N\}$, where $0=t_{0}<t_{1}<\cdots<t_{N-1}<t_{N}=T$ such that $I=\cup_{n=1}^{N}I_{n}$. By $\t_n$ we denote the time step size of the interval $I_{n}$, i.e. $\t_n=t_{n}-t_{n-1}$, and the maximum time step size by $\t=\max_{n=1,...,N}\{\t_n\}$.
	For a Banach space $B$ and any $k\in\mathbb{N}_{0}$, we define the space
	\begin{align}\label{space:PK}
		\mathbb{P}_{k}(I_n;B):= \left\{w_\t:I_n\to B, w_{\t}(t)=\sum_{j=0}^{k}W^{j}t^{j},\,\forall t\in I_{n}, W^{j}\in B \,\forall\, j\right\}.
	\end{align}
	For an integer $k\in\mathbb{N}$ and $l\in\mathbb{N}_{0}$ we define the spaces of piecewise polynomials of order $k$ and $l$ in time 
	\begin{align}\label{space:globalcontTIME}
		X_{\t}^{k}(B)&:=\left\{w_{\t}\in C(\bar{I};B)\,|\,w_{\t|I_{n}}\in\P_k(I_n;B),\,\forall \,n\in\{1,...,N\}\right\},\\
		Y_{\t}^{l}(B)&:=\left\{w_{\t}\in L^{2}(I;B)\,|\,w_{\t|I_{n}}\in\P_l(I_n;B),\,\forall n\in\{1,...,N\}\right\},
	\end{align}
	i.e., the space of globally continuous in time functions and global $L^{2}$ functions in time.
	We also define the one-sided limits of a piecewise continuous function $w: I\to B$ with respect to the time mesh $\{I_1,...,I_n\}$ by
	\begin{align}\label{eq:limitdef}
		w(t_n^{+}):=\lim_{t\to t_n+0}w(t),\mbox{ for }n<N, \mbox{ and }\quad w(t_n):=\lim_{t\to t_n-0}w(t),\mbox{ for } n>0.
	\end{align}

	In space, we consider a quasi-uniform and shape-regular mesh $\mathcal{T}_h$ of the domain $\Om$. Let $K$ be an element of the partition $\mathcal{T}_h$ and $h_{K}$ the corresponding diameter of element $K$.  By $h:=\max_{K\in\mathcal{T}_{h}}\{h_K\}$ we denote the maximum element diameter of the partition.  Then for any $r\in\N$, we define the finite element space $V_h^{r}$ built on simplexes $K$ by 
	\begin{align}
		V^{r}_{h}:=\left\{v_h\in C(\widebar{\Om}) \,|\,v_{h|k}\in\mathbb{P}_r(K)\,\forall \, K\in\mathcal{T}_h\right\}\cap H_{0}^{1}(\Om),
	\end{align}
	where $\P_{r}(K)$ is the space defined by the reference mapping of polynomials on the reference element with maximum degree $r$ in each variable. For vector-valued functions, we write the space $V_h^r$ in bold. 
	
	\subsection{Auxilliaries}
	We will consider the $(k+1)$-point Gauss-Lobatto quadrature rule for the continuous in-time finite element method since it preserves the continuity at both ends of the interval $I_n$. We define an affine transformation from the reference interval $\hat{I}=[-1,1]$ to $\bar{I}_n$ by
	\begin{align*}
		T_{n}(\hat{t}):=\frac{t_{n-1}+t_{n}}{2}+\frac{\t_{n}}{2}\hat{t},
	\end{align*}
	and we let $\hat{t}_{\mu}$ and $\hat{t}_{\mu}^{G}$ for $\mu=0,1,...,k$ be the Gauss-Lobatto and Gauss quadrature points on the reference interval $\hat{I}$ respectively.
	Then the $(k+1)$-point Gauss-Lobatto quadrature  formula on each time interval $I_n$ is
	\begin{align}\label{Gauss-Lobatto}
		Q_{n}(w):=\frac{\t}{2}\sum_{\mu=0}^{k}\hat{\om}_{\mu}w_{|I_n}(t_{n,\mu})\approx\int_{I_n}w(t)dt,
	\end{align}
	where $t_{n,\mu}=T_n({\hat{t}_{\mu}})$, for $\mu=0,1,...,k$ are the quadrature points on $\bar{I}_{n}$ and $\hat{\om}_{\mu}$ is the corresponding weights.  Also the $(k)$-point Gauss quadrature formula on each time interval $I_n$ is 
	\begin{align}\label{gaussquadratureInKthpoint}
		Q_{n}^{G}:=\frac{\t}{2}\sum_{\mu=1}^{k}\hat{\om}_{\mu}^{G}w(t_{n,\mu}^{G})\approx \int_{I_n}w(t)\,dt,
	\end{align}
	where $\glp^{G}=T_{n}(\hat{t}_{\mu}^{G})$ for $\mu=1,...,k$, are the Gauss quadrature points on $I_n$ and $\hat{\om}_{\mu}^{G}$ the corresponding weights.
	Note that for all polynomials in $\P_{2k-1}(I_n;\R)$, the quadrature formulas \eqref{Gauss-Lobatto} and \eqref{gaussquadratureInKthpoint} are exact.
	
	In the analysis, we are going to need a temporal interpolation operator. We define the global Lagrange interpolation operator $I_{\t}:C^{0}(\bar{I};L^{2})\mapsto X^{k}_{\t}(L^{2})$ for the Gauss-Lobatto quadrature points $\glp$ as
	\begin{align}\label{operator:globalLagrangeinterpolation}
		I_{\t}w(\glp)=w(\glp), \mbox{ for }\,\mu=0,....,k,\,\,n=1,...,N.
	\end{align}
	We also define the local interpolant $I_{\t,n}^{G}:C^{0}(\bar{I}_n;L^{2})\mapsto\P_{k-1}(\bar{I}_n;L^{2})$
	for $n=1,...,N$, by 
	\begin{align}\label{localinterpolant}
		I^{G}_{\t,n}w(\glp^{G})=w(\glp^{G}),\,\,\mu=1,...,k.
	\end{align}
	Further, we define the interpolate $\Pi_{\t}^{k-1}w\in Y^{k-1}_{\t}(L^{2})$ for a given function $w\in L^{2}(I;L^{2})$ such that the restriction of $\Pi_{\t}^{k-1}w$ to $I_n$, i.e. $\Pi_{\t}^{k-1}w_{|I_n}\in\P_{k-1}(I_n;L^{2}),\,$ for $n=1,...,N$, is determined by the local $L^{2}$-projection in time,
	\begin{align}
		\int_{I_n}\langle\Pi_{\t}^{k-1}w,q\rangle \,dt=\int_{I_n}\langle w,q\rangle \,dt\quad \forall\, q\in\P_{k-1}(I_n;L^{2}). 
	\end{align}

	\begin{lemma}\label{lemma:21}
		Consider the Gauss quadrature formula \eqref{gaussquadratureInKthpoint}. For all polynomials $w\in\P_{k}(I_n;L^{2})$ it holds that
		\begin{subequations}\label{eqs:lemma2_1}
			\begin{align}
				\Pi_{\t}^{k-1}w(t)&=I_{\t,n}^{G}w(t),\quad \mbox{for }t\in I_n, \,n=1,...,N,\\
				\Pi_{\t}^{k-1}w(\glp^{G})&=w(\glp^{G}),\quad \,\,\mbox{for }\mu=1,...,k,\,n=1,...,N.
			\end{align}
		\end{subequations}
	\end{lemma}
	\begin{proof}
		See \cite[Eq.\ 2.6]{makridakis2005} and \cite[Lemma\ 4.5]{bause2020}.
	\end{proof}
	We recall the following $L^{\infty}-L^{2}$ relation from \cite[Eq. 2.5]{makridakis2005}.
	\begin{lemma}
		For all $n=1,...,N$ it holds that
		\begin{align}\label{eq:LInfinityLTwo}
			\|w\|_{L^{\infty}(I_n;\R)}\leq c\t^{-\frac{1}{2}}\|w\|_{L^{2}(I_n;\R)},\quad \mbox{for all }w\in\P_{k}(I_n;L^{2}).
		\end{align}
	\end{lemma}

	\section{Discretization}\label{sec:disc}
	
	\subsection{Continuous system and stability estimates}
	
	For the numerical approximation, we recast the system \eqref{eqs:biotallardADE} as a first-order system in time. For this, we define $\bmv:=\p_t\bmu$ and obtain the model
	\begin{subequations}\label{eqs:firstorderADE}
		\begin{align}
			\p_t\bmu-\bmv&=0,\quad&&\mbox{ in }\Om\times(0,T],\\
			\rho\p_t\bmv -\div\bm{C}\e(\bmu) +\a \del p+\rho_f\p_t \Psi&=\rho\bmF,\quad&&\mbox{ in }\Om\times(0,T],\\
			c_0\p_t p +\a \div \bm{v}+\div\Psi&=0,\quad&&\mbox{ in }\Om\times(0,T],\\
			c_1\rho_f F\p_t\Psi +\rho_f F\Psi+d_1\eta_k\del p+\rho_fd_1\eta_k\p_t\bmv&=\rho_fd_1\eta_k\bmF,\quad&&\mbox{ in }\Om\times(0,T],\\
			\bmu(0)=\bmu_0,\quad\bmv(0)=\bmv_0, \quad p(0)=p_0,\quad\Psi(0)&=\Psi_0, \quad&&\mbox{ in }\Om,
			\\
			\bmu=\bm{0},\quad\bmv=\bm{0}, \quad p=0, \quad \Psi &= 0\quad&&\mbox{ on }\p\Om\times(0,T].
		\end{align}
	\end{subequations}
	The existence of a sufficiently smooth solution to \eqref{eqs:firstorderADE} such that all of the below-given estimates are well-defined is assumed here. For this, the data are supposed to be sufficiently smooth as well. The corresponding weak formulation reads as follows:  Find $(\bmu,\bmv, p, \Psi)\in \left(L^{2}(I;\bm{H}^{1}_{0})\times L^{2}(I;\bm{H}^{1}_{0})\times L^{2}(I;H^{1}_{0})\times L^{2}(I;\bm{H}^{1}_{0})\right)$ such that
	\begin{subequations}\label{eqs:WeakformfirstorderADE}
		\begin{align}
			\begin{split}
				\int_0^{T}\langle\p_t\bmu,\bmPh^{U}\rangle dt - \int_0^{T}\langle\bmv,\bmPh^{U}\rangle\,dt=0,\label{eq:appUV}
			\end{split}
			\\
			\begin{split}
				\int_0^{T}\rho\langle\p_t\bmv,\bmPh^{V}\rangle dt +\int_0^{T}\langle\bm{C}\e(\bmu),\e(\bmPh^{V})\rangle dt+\int_0^{T}\a \langle\del p,\bmPh^{V}\rangle dt\\+\int_0^{T}\rho_f\langle\p_t \Psi,\bmPh^{V}\rangle\, dt=\int_0^{T}\rho\langle \bmF,\bmPh^{V}\rangle \,dt,
			\end{split}
			\\
			\begin{split}
				\int_0^{T}c_0\langle\p_t p,\Phi^{P}\rangle dt +\int_0^{T}\a \langle\div \bm{v},\Phi^{P}\rangle dt+\int_0^{T}\langle\div\Psi,\Phi^{P}\rangle\, dt=0,\label{eq: appPressure}
			\end{split}
			\\
			\begin{split}
				\int_0^{T} c_1\rho_f F\langle\p_t\Psi,\bmPh^{\Psi}\rangle dt+\int_0^{T}\rho_f F\langle\Psi,\bmPh^{\Psi}\rangle dt+ \int_0^{T}d_1\eta_k\langle\del p,\bmPh^{\Psi}\rangle dt\\+\int_0^{T}\rho_fd_1\eta_k\langle\p_t\bmv,\bmPh^{\Psi}\rangle \,dt=\int_0^{T}\rho_fd_1\eta_k\langle \bmF,\bmPh^{\Psi}\rangle\, dt,
			\end{split}
		\end{align}
		for $(\bmPh^{U},\bmPh^{V},\Phi^{P},\bmPh^{\Psi})\in \left(L^{2}(I;\bm{H}^{1}_{0})\times L^{2}(I;\bm{H}^{1}_{0})\times L^{2}(I;H^{1}_{0})\times L^{2}(I;\bm{H}^{1}_{0})\right)$.
	\end{subequations}
	
	Next, we still derive a stability estimate for the solution of \eqref{eqs:WeakformfirstorderADE}. From this result, we identify the first-order energy of the system's unknowns $(\bmu,\bmv, p, \Psi)$. The stability result will then induce the norm for our error control that is based on a corresponding estimation of the discrete error between an interpolation of the continuous solution and its fully discrete approximation.
	
	For simplicity, we derive the stability estimate assuming that the source terms $\bmF\equiv 0$. In \eqref{eqs:WeakformfirstorderADE}, we choose the test functions $\bmPh^{U}=\p_t\bmu$, $\bmPh^{V}=\bmv$, $\Phi^{P}=p$ and $\bmPh^{\Psi}=\frac{1}{d_1\eta_k}\Psi$ based on seeking to cancel out the coupled terms.  First note that \eqref{eq:appUV} disappears, and when we add the equations in \eqref{eqs:WeakformfirstorderADE} together, the coupling terms arising from $\a\del p$ and $\a\div\bmv$ cancel each other out. The same happens for the coupling terms between the auxiliary and pressure equations. Then, we are left with two more coupling terms
	$$\int_0^{T}\rho_f\langle\p_t\bmv,\Psi\rangle \,dt\quad\mbox{ and }\quad\int_{0}^{T}\rho_f\langle\p_t \Psi,\bmv\rangle \,dt.$$
To control these terms, we derive an estimate for the velocity and the memory variable simultaneously. Because of our chosen test functions, we define the matrix
\begin{align}\label{eq:Nmatrix}
	\bm{N}:= \begin{pmatrix}
		\rho&\rho_f\\
		\rho_f&\rho_fc_1 F\eta_k^{-1}d_1^{-1}\
	\end{pmatrix},
\end{align}
and write the combined weak form as
\begin{align*}
	\int_0^{T}\llangle[\bigg]\bm{N}\begin{pmatrix}
		\p_t\bmv\\
		\p_t\Psi
	\end{pmatrix},\begin{pmatrix}
		\bmv\\
		\Psi
	\end{pmatrix}\rrangle[\bigg]+\langle\bm{C}\e(\bmu),\e(\bmv)\rangle+c_0\langle\p_t p,p\rangle +
	\frac{\rho_f F}{d_1\eta_k}\|\Psi\|^{2}_{\bm{L}^{2}} \,dt =0.
\end{align*}
This then results in the stability estimate
\begin{equation}
	\begin{aligned}\label{eq:errorstabilityestimate}
		\frac{1}{2}\llangle[\bigg]\bm{N}\begin{pmatrix}
			\bmv(T)\\
			\Psi(T)
		\end{pmatrix},\begin{pmatrix}
			\bmv(T)\\
			\Psi(T)
		\end{pmatrix}\rrangle[\bigg]+\frac{1}{2}\langle \textbf{C}\e(\bmu(T)),\e(\bmu(T))\rangle+\frac{c_0}{2}\langle p(T),p(T)\rangle \\\leq
		\frac{1}{2}\llangle[\bigg]\bm{N}\begin{pmatrix}
			\bmv_0\\
			\Psi_0
		\end{pmatrix},\begin{pmatrix}
			\bmv_0\\
			\Psi_0
		\end{pmatrix}\rrangle[\bigg]+\frac{1}{2}\langle \textbf{C}\e(\bmu_0),\e(\bmu_0)\rangle+\frac{c_0}{2}\langle p_0,p_0\rangle.
	\end{aligned}
\end{equation}
Since the solution of the weak form at a final time $T$ is bounded by the initial conditions in this combined energy quantity, we have chosen to analyze the convergence of the numerical scheme with respect to this quantity. \begin{remark} 
		The matrix $\bm{N}$ defined in \eqref{eq:Nmatrix} is symmetric positive definite if $\frac{\rho c_1 F}{\eta_k d_1}-\rho_f>0$. Using the definition of the kinetic viscosity, this simplifies to $\frac{\rho c_1 F}{\eta d_1}>1$. First, note that for any poroelastic medium, the formation factor satisfies $F\geq1$ and becomes large for materials with low porosity.  Also, for realistic parameter values, the total mass density $\rho>1$, and viscosity $\eta<1$. Further, we note that $c_1>d_1$ for a variety of different poroelastic materials, including sandstone saturated with water, epoxy-glass, and materials with higher porosity, cf. \cite[Remark 1]{stokke2024}. Further examples where $c_1>d_1$ can be found in Appendix \ref{appB}. In essence, this means that $\bm{N}$ is SPD when considering realistic parameter values.\end{remark}
\subsection{Bilinear forms and discrete operators}
Here, we introduce the bilinear forms for the discrete variational formulation and more operators related to the spatial discretization.
For $\bmu,\bmv,\bm{\phi}\in\bm{H}_0^{1},p,\psi\in H^{1}_{0}$ and $\bm{f}\in H^{-1}$ we define 
\begin{align*}
	A(\bmu,\bm{\phi}):=&\langle\textbf{C}\e(\bmu),\e(\bm{\phi})\rangle,\quad &&C(\bmv,\psi):=-\a\langle\div\bmv,\psi\rangle,\\
	\quad G_1(\bm{\phi}):=&\langle\rho\bm{f},\bm{\phi}\rangle, \quad  &&G_2(\bm{\phi}):=\langle\rho_f d_1\eta_k\bm{f},\bm{\phi}\rangle. 
\end{align*}
We denote the $\bm{L}^{2}$-orthogonal projection onto $\bm{V}^{r}_{h}$ by $\bm{P}_{h}:\bm{L}^{2}\mapsto\bm{V}^{r}_{h}$ such that for $\bmw\in\bm{L}^{2}$ we have $\langle\bm{P}_{h}\bmw,\bm{\phi}_{h}\rangle=\langle\bmw,\bm{\phi}_{h}\rangle$ for all $\bm{\phi}_{h}\in\bm{V}^{r}_{h}$. The operator $\bm{R}_{h}:\bm{H}_{0}^{1}\mapsto\bm{V}^{r}_{h}$ defines the elliptic projection onto $\bm{V}_{h}^{r}$ such that 
\begin{align*}
	\langle \textbf{C}\e(\bm{R}_{h}\bmw)-\textbf{C}\e(\bmw),\e(\bm{\phi}_{h})\rangle=	0, \quad \mbox{for }\bmw\in H^{1}_{0},\,\bm{\phi}_{h}\in\bm{V}^{r}_{h}.
\end{align*} 
Let the relation $\langle\bm{A}_{h}\bmw,\bm{\phi}_{h}\rangle=A(\bmw,\bm{\phi})$ hold for the discrete operator $\bm{A}_{h}:\bm{H}_0^{1}\mapsto\bm{V}^{r}_{h}$ for all $\bm{\phi}_{h}\in\bm{V}^{r}_{h}$. It follows that for $\bmw\in\bm{H}_0^{1}\cap\bm{H}^{2}$
\begin{align*}
	\langle\bm{A}_{h}\bmw,\bm{\phi}_{h}\rangle=\langle \textbf{C}\e(\bmw),\e(\bm{\phi}_{h})\rangle=\langle\bm{A}\bmw,\bm{\phi}_{h}\rangle,\,\mbox{for }\bm{\phi}\in\bm{V}^{r}_{h},\, 
\end{align*}
$\mbox{where }\bm{A}:\bm{H}_{0}^{1}\mapsto\bm{H}^{-1} \mbox{ with }\langle\bm{A}\bmw,\bm{\phi}\rangle:=A(\bmw,\bm{\phi}),\,\mbox{for }\bm{\phi}\in\bm{H}_0^{1}.$
Therefore $\bm{A}_{h}\bmw=\bm{P}_{h}A\bmw$ for $\bmw\in\bm{H}^{1}_{0}\cap\bm{H}^{2}$. By $R_h$ we denote the elliptic projection onto $V^{r}_h$ such that 
\begin{equation*}
	\langle \del R_h p-\del p, \varphi_h\rangle = 0, \quad\forall \varphi_h\in V_h^{r}.
\end{equation*}

\subsection{Continuous Galerkin discretization}
\label{Subsec:CGD}
In space, we use a continuous finite element method based on equal-order elements, and in time we use the continuous Galerkin method (see e.g. \cite{ern2004,brenner2008}). This gives the following discrete problem on the subinterval $I_n$ of \eqref{eqs:firstorderADE},
given that solutions have been computed at $t_{n-1}$.
\begin{problem}[Quadrature form of $I_n$ problem]\label{prob:quadratureform}
	\label{Prob1}
	Let $k,r\geq 1$. For given $\bmu^{n-1}_{\t,h}:=\bmu_{\t,h}(t_{n-1})\in\bm{V}_{h}^{r}, \bmv_{\t,h}^{n-1}:=\bmv_{\t,h}(t_{n-1})\in\bm{V}_{h}^{r}, p_{\t,h}^{n-1}:=p_{t,h}(t_{n-1})\in V_{h}^{r},\Psi_{\t,h}^{n-1}:=\Psi_{\t,h}(t_{n-1})\in\bm{V}_{h}^{r}$ with initial values denoted by setting $\t=0$. Find $(\bmu_{\t,h},\bmv_{\t,h},p_{\t,h},\Psi_{t,h})\in (\P_k(I_n;V_h^{r}))^{d}\times (\P_k(I_n;V_h^{r}))^{d}\times \P_k(I_n;V_h^{r})\times(\P_k(I_n;V_h^{r}))^{d}$ such that $\bmu_{\t,h}(t_{n-1})=\bmu^{n-1}_{\t,h},\bmv_{\t,h}(t_{n-1})=\bmv^{n-1}_{\t,h},p_{\t,h}(t_{n-1})=p^{n-1}_{\t,h},\Psi_{\t,h}(t_{n-1})=\Psi^{n-1}_{\t,h}$ and
	\begin{subequations}
		\begin{align}
			\begin{split}
				Q_{n}\left(\langle\p_t\bmu_{\t,h},\bm{\Phi}_{\t,h}^{U}\rangle\right)-Q_{n}\left(\langle\bmv_{\t,h},\bm{\Phi}_{\t,h}^{U}\rangle\right)=0,
			\end{split}
			\\
			\begin{split}
				Q_{n}\left(\langle \rho\p_t\bmv_{\t,h},\bm{\Phi}_{\t,h}^{V}\rangle\right)+Q_{n}\left(A(\bmu_{\t,h},\bm{\Phi}_{\t,h}^{V})\right)+Q_{n}\left(C(\bm{X}_{\t,h},p_{\t,h})\right)\\+Q_{n}\left(\langle\rho_f\p_t\Psi,\bm{\Phi}_{\t,h}^{V}\rangle\right)=Q_{n}\left(G_1(\bm{\Phi}_{\t,h}^{V})\right),
			\end{split}
			\\
			\begin{split}
				Q_{n}\left(\left\langle c_0\p_tp_{\t,h},\Phi_{\t,h}^{P}\right\rangle\right)-Q_{n}\left(C(\p_t\bmu_{\t,h},\Phi_{\t,h}^{P})\right)+Q_{n}\left(\left\langle \div\Psi_{\t,h},\Phi_{\t,h}^{P}\right\rangle\right)=0,
			\end{split}
			\\
			\begin{split}
				Q_{n}\left(\left\langle\rho_fc_1F\p_t\Psi_{\t,h},\bm{\Phi}_{\t,h}^{\Psi}\right\rangle\right)
				+Q_{n}\left(\left\langle\rho_fF\Psi_{\t,h},\bm{\Phi}_{\t,h}^{\Psi}\right\rangle\right)-Q_{n}\left(\frac{d_1\eta_k}{\a}C(\bm{\Phi}_{\t,h}^{\Psi},p_{\t,h})\right)\\+Q_{n}\left(\left\langle\rho_fd_1\eta_k\p_t\bmv_{\t,h},\bm{\Phi}_{\t,h}^{\Psi}\right\rangle\right)=Q_{n}\left(G_2(\bm{\Phi}_{\t,h}^{\Psi})\right),
			\end{split}
		\end{align}
		for $\left(\bm{\Phi}_{\tau,h}^{U},\bm{\Phi}_{\tau,h}^{V},\Phi_{\tau,h}^{P},\bm{\Phi}_{\tau,h}^{\Psi}\right)\in (\P_{k-1}(I_n;V_h^{r}))^{d}\times (\P_{k-1}(I_n;V_h^{r}))^{d}\times \P_{k-1}(I_n;V_h^{r})\times(\P_{k-1}(I_n;V_h^{r}))^{d}$.
	\end{subequations}
\end{problem}

\begin{problem}[Variational form of $I_n$ problem]
	\label{Prob2}
	Let $k,r\geq 1$. For given $\bmu^{n-1}_{\t,h}:=\bmu_{\t,h}(t_{n-1})\in\bm{V}_{h}^{r}, \bmv_{\t,h}^{n-1}:=\bmv_{\t,h}(t_{n-1})\in\bm{V}_{h}^{r}, p_{\t,h}^{n-1}:=p_{t,h}(t_{n-1})\in V_{h}^{r},\Psi_{\t,h}^{n-1}:=\Psi_{\t,h}(t_{n-1})\in\bm{V}_{h}^{r}$ with initial values denoted by setting $\t=0$. Find $(\bmu_{\t,h},\bmv_{\t,h})\in (\P_k(I_n;V_h^{r}))^{d}\times (\P_k(I_n;V_h^{r}))^{d}$ and $(p_{\t,h},\Psi_{t,h})\in  \P_k(I_n;V_h^{r})\times(\P_k(I_n;V_h^{r}))^{d}$ such that $\bmu_{\t,h}(t_{n-1})=\bmu_{\t,h}^{n-1},\bmv_{\t,h}(t_{n-1})=\bmv_{\t,h}^{n-1}), p_{\t,h}(t_{n-1})=p^{n-1}_{\t,h}, \Psi_{\t,h}(t_{n-1})=\Psi^{n-1}_{\t,h}$ and
	\begin{subequations}\label{variational}
		\begin{align}
			&\int_{I_n}\langle \p_t \bmu_{\t,h},\bmPth^{1}\rangle dt- \int_{I_n}\langle  \bmv_{\t,h},\bmPth^{1}\rangle \,dt= 0,\\
			\begin{split}
				&\int_{I_n}\langle \p_t \bmv_{\t,h},\bmPth^{2}\rangle dt+\int_{I_n}\langle \bm{A}_{h}\bmu_{\t,h},\bmPth^{2}\rangle dt-\int_{I_n}\a\langle p_{\t,h},\div\bmPth^{2}\rangle dt\\
				&+\int_{I_n}\langle \rho_f\p_t\Psi_{\t,h},\bmPth^{2}\rangle\, dt = Q_{n}\left(\langle \rho\bmF,\bmPth^{2}\rangle\right),
			\end{split}
			\\
			\begin{split}
				&\int_{I_n}\langle c_0\p_tp_{\t,h},\psi_{\t,h}\rangle dt+\int_{I_n}\a\langle\div\p_t\bmu_{\t,h},\psi_{\t,h}\rangle dt+\int_{I_n}\langle\div\Psi_{\t,h},\psi_{\t,h}\rangle\, dt=0,
			\end{split}
			\\
			\begin{split}
				&\int_{I_n}\langle \rho_fc_1F\p_t\Psi_{\t,h},\bmPth^{3}\rangle dt+\int_{I_n}\langle \rho_fF\Psi_{\t,h},\bmPth^{3}\rangle dt- \int_{I_n}\langle d_1\eta_kp_{\t,h},\div\bmPth^{3}\rangle dt
				\\
				&+\int_{I_n}\langle \rho_fd_1\eta_k\p_t\bmv_{\t,h},\bmPth^{3}\rangle \,dt=Q_{n}\left(\left\langle\rho_fd_1\eta_k\bmF,\bmPth^{3}\right\rangle\right)
			\end{split}
		\end{align}
	\end{subequations}
	for all $(\bmPth^{1},\bmPth^{2}, \psi_{\t,h},\bmPth^{3})\in(\P_{k-1}(I_n;V_h^{r}))^{d}\times(\P_{k-1}(I_n;V_h^{r}))^{d}\times\P_{\t,h}(I_n;V_h^{r})\times(\P_{k-1}(I_n;V_h^{r}))^{d}$.
\end{problem}

\section{Error analysis}\label{sec:error}
In this section, we derive the error estimate \eqref{eq:introerror} for the discrete \Cref{prob:quadratureform}. The key ideas and steps of the derivation of \eqref{eq:introerror} are illustrated in the flow chart of \Cref{fig:schematic}.

\subsection{Auxiliary results for the error analysis}
Here, we present some auxiliary results needed for the error analysis. 

\begin{definition}[Special approximation $(\bmw_{1},\bmw_{2})$ of $(\bmu,\p_t\bmu)$]\label{def:specialapprox}
	Let $\bmu\in C^{1}(\widebar{I};\bm{H}_0^{1})$ be given. On the subinterval $I_{n}=(t_{n-1},t_{n}]$ we define
	\begin{align}\label{eq:special approximation}
		\bmw_1:=I_{\t}\left(\int_{t_{n-1}}^{t}\bmw_{2}(s)\,ds+\bm{R}_{h}\bmu(t_{n-1})\right),\quad \mbox{ where } \bmw_2:=I_{\t}(\bm{R}_h\p_t\bmu).
	\end{align}
	Also $\bmw_1(0):=\bm{R}_h\bmu(0)$.
\end{definition}
\begin{lemma}
	For $\bmw_{1}$ and $\bmw_2$ defined in \cref{def:specialapprox}, it holds that 
	\begin{align}
		\int_{I_n}\langle\p_t\bmw_1,\bm{\phi}_{\t,h}\rangle \,dt = \int_{I_n}\langle\bmw_2,\bm{\phi}_{\t,h}\rangle\, dt,\quad\forall\,\bm{\phi}_{\t,h}\in\left(\P_{k-1}(I_n;V^{r}_{h})\right)^{d}.\label{eq:lemma:4_2}
	\end{align}
\end{lemma}
\begin{lemma}\label{lemma:4_3}
	For $\bm{y}_{\t,h},\bm{z}_{\t,h} \in\left(\P_{k-1}(I_n;V^{r}_{h})\right)^{d}$ let
	\begin{align}
		\int_{I_n}\langle\p_t\bm{y},\bm{\phi}_{\t,h}\rangle-\langle\bm{z},\bm{\phi}_{\t,h}\rangle \,dt=0,\quad\forall\,\bm{\phi}_{\t,h}\in\left(\P_{k-1}(I_n;V^{r}_{h})\right)^{d}.
	\end{align}
	Then, it holds that 
	\begin{align}
		\p_t\bm{y}_{\t,h}(t^{G}_{n,\mu})=\bm{z}_{\t,h}(t^{G}_{n,\mu}), \quad \mbox{ for }\mu=1,...,k,
	\end{align}
	where $t^{G}_{n,\mu}$ are the Gauss quadrature nodes of the subinterval $I_n$.
\end{lemma}
A useful result is that the Lagrange interpolant satisfies the stability estimates (cf.\ \cite[Eqs. (3.15) and (3.16)]{makridakis2005})
\begin{subequations}\label{eqs:stability}
	\begin{align}\label{eq:stabilitya}
		\|I_{\t}w\|_{L^{2}(I_n;L^{2})}\leq&\, c\|w\|_{L^{2}(I_n;L^{2})}+c\t_{n}\|\p_t w\|_{L^{2}(I_n;L^{2})},\\
		\left\|\int_{t_{n-1}}^{t}w ds\right\|_{L^{2}(I_n;L^{2})}\leq&\,
		c\t_{n}\|w\|_{L^{2}(I_n;L^{2})}\label{eq:stabilityb}.
	\end{align}
\end{subequations}
We will also make use of the $H^{1}-L^{2}$ inverse inequality
\begin{align}\label{eq:HoneLtwoInverseIneq}
	\|w'\|_{L^{2}(I_n;\R)}\leq c\t_{n}^{-1}\|w\|_{L^{2}(I_n;\R)},\quad \forall\,w\in\P_{k}(I_{n};\R).
\end{align}

\subsection{Splitting of the error}
First, we split the error into 
\begin{subequations}\label{eqs:errorsplit}
	\begin{align}
		\bmu-\bmu_{\t,h}&=\bmu-\bm{w}_1+\bm{w}_1-\bmu_{\t,h}:=\bm{\eta}_1+\bmEth^{1}\\
		\bmv-\bmv_{\t,h}
		&=
		\bmv-\bm{w}_2+
		\bm{w}_2-\bmv_{\t,h}
		=:\bm{\eta}_2+\bm{E}_{\t,h}^{2},\\
		p-p_{\t,h}&=p-I_\t R_h p+I_\t R_h p-p_{\t,h}=:\om+e_{\t,h},\label{eq:pressureerrorsplit}\\
		\Psi-\Psi_{\t,h}&=\Psi-I_\t \bmR_h \Psi+I_\t \bmR_h \Psi-\Psi_{\t,h}=:\bm{\zeta}+\bm{E}_{\t,h}^{3},\label{eq:adeerrorsplit}
	\end{align}
\end{subequations}
where $(\bmw_1,\bmw_2)$ are defined in \Cref{def:specialapprox}. We refer to $\bmEth^{i}$ and $e_{\t,h}$ as the discrete error and $\bm{\eta}_i$, $\bm{\zeta}$ and $\om$ as the interpolation error.
We define a norm $\norm{\cdot}$ and a weighted elasticity norm $\norm{\cdot}_e$ for a quantity $\bm{Z}\in\bm{H}_0^{1}$ by
\begin{align}
	\norm{\bm{Z}}:=\left(\|\del\bm{Z}\|^{2}\right)^{1/2},\quad 
	\norm{\bm{Z}}_e:=\left(\frac12\langle\textbf{C}\e(\bm{Z}),\e(\bm{Z})\rangle\right)^{1/2}.
\end{align}
These norms are equivalent in the sense that there holds that 
\begin{align}\label{eq:norm}
	c_1\norm{\bm{Z}}\leq\norm{\bm{Z}}_e\leq c_2\norm{\bm{Z}}, \quad c_1>0,\quad c_2>0.
\end{align}
Recall from \cite{howell1991} that the Lagrange interpolation \eqref{operator:globalLagrangeinterpolation} satisfies
\begin{equation}\label{eq:lagrangestability}
	\|f-I_\t f\|_{L^{q}(I_n;H^{m})}\leq c\t_{n}^{k+1}\|\p_t^{k+1}f\|_{L^{q}(I_n;H^{m})}, \quad q\in\{2,\infty\},\quad m\in\{0,1\}.
\end{equation}
In addition for the elliptic projections onto $\bm{V}^{r}_h$ and $V^{r}_{h}$ we have the error estimates (see e.g. \cite{brenner2008})
\begin{subequations}
	\begin{align}
		\|p-R_h p\|+h\|\del(p-R_h p)\|\leq c h^{r+1}\|p\|_{r+1},\label{eq:ellipticprojectionerrorScalar}\\
		\|\bmv-\bmR_h\bmv\|+h\|\del(\bmv-\bmR_h\bmv)\|\leq c h^{r+1}\|\bmv\|_{r+1}.\label{eq:ellipticprojectionerrorVector}
	\end{align}
\end{subequations}
We assume that the initial approximations $\bmu_{0,h},\bmv_{0,h}, p_{0,h}$ and $\Psi_{0,h}$ are chosen such that the approximation properties
\begin{subequations}\label{ass:initialapprox}
	\begin{align}
		\|\del(\bmR_h\bmu_0-\bmu_{0,h})\|\leq&\, ch^{r}\|\bmu_0\|_{r+1},\\
		\|\bmR_h\bmv_0-\bmv_{0,h}\|\leq&\, ch^{r+1}\|\bmv_0\|_{r+1},\\
		\|R_hp_0-p_{0,h}\|\leq&\, ch^{r+1}\|p_0\|_{r+1},\\
		\|\bmR_h\Psi_0-\Psi_{0,h}\|\leq&\, ch^{r+1}\|\Psi_0\|_{r+1}.
	\end{align}
\end{subequations}

\begin{lemma}\label{lemma:splittingestimate}
	For the quantities defined when splitting the error in \eqref{eqs:errorsplit}, there holds
	\begin{subequations}
		\begin{align}
			\begin{split}
				\|\bm{\eta}_1\|_{L^{m}(I_n;\bm{L}^{2})}\leq&\,c\Big(\t_{n}^{k+1}\|\p_t^{k+1}\bmu\|_{L^{m}(I_n,\bm{L}^{2})}+\t_{n}^{k+1}\|\p_t^{k+2}\bmu\|_{L^{m}(I_n,\bm{L}^{2})}
				\\&\quad+h^{r+1}\|\bmu\|_{L^{m}(I_n;\bm{H}^{r+1})}+\t_{n}h^{r+1}\|\p_t\bmu\|_{L^{m}(I_n;\bm{H}^{r+1})}\\&\quad\quad+\t_n^{2}h^{r+1}\|\p_t^{2}\bmu\|_{L^{m}(I_n;\bm{H}^{r+1})}\Big),
			\end{split}
			\\
			\begin{split}
				\|\bm{\eta}_2\|_{L^{m}(I_n;\bm{L}^{2})}\leq&\,c\Big(\t_{n}^{k+1}\p_t^{k+2}\bmu\|_{L^{m}(I_n,\bm{L}^{2})}+h^{r+1}\|\p_t\bmu\|_{L^{m}(I_n;\bm{H}^{r+1})}\\
				&\quad+\t_nh^{r+1}\|\p_t^{2}\bmu\|_{L^{m}(I_n;\bm{H}^{r+1})}\Big), 
			\end{split}
			\\
			\begin{split}
				\|\bm{\eta}_1\|_{L^{m}(I_n;\bm{H}^{1})}\leq&\,c\Big(\t_{n}^{k+1}\|\p_t^{k+1}\bmu\|_{L^{m}(I_n,\bm{H}^{1})}+\t_{n}^{k+2}\|\p_t^{k+2}\bmu\|_{L^{m}(I_n,\bm{H}^{1})}\\
				&\quad+h^{r}\|\bmu\|_{L^{m}(I_n;\bm{H}^{r+1})}\Big), \label{eq:splitting2}
			\end{split}
			\\
			\begin{split}
				\|\om\|_{L^{m}(I_n;L^{2})}\leq&\,c\Big(\t_n^{k+1}\|\p_t^{k+2}p\|_{L^{m}(I_n,L^{2})}+h^{r+1}\|p\|_{L^{m}(I_n;H^{r+1})}\\&\quad+\t_nh^{r+1}\|\p_tp\|_{L^{m}(I_n;H^{r+1})}\Big),\label{eq:splitting3} 
			\end{split}
			\\
			\begin{split}
				\|\bm{\zeta}\|_{L^{m}(I_n;\bm{L}^{2})}\leq &\,c\Big(\t_n^{k+1}\|\p_t^{k+1}\Psi\|_{L^{m}(I_n;\bm{L}^{2})}+h^{r+1}\|\Psi\|_{L^{m}(I_n;\bm{H}^{r+1})}\\&\quad+\t_nh^{r+1}\|\p_t\Psi\|_{L^{m}(I_n;\bm{H}^{r+1})}\Big),\label{eq:splitting4}
			\end{split}
		\end{align}
	\end{subequations}
	for $m=2$ or $m=\infty$.
\end{lemma}
\begin{proof}
	The first two inequalities follow similarly to \cite[Lemma 3.3]{makridakis2005}. \eqref{eq:splitting2} is proved for the scalar-valued case in \cite[Appendix]{bause2020}. It also holds for the vector-valued case. Consider \eqref{eq:splitting4}. From the error splitting \eqref{eqs:errorsplit}, we have $\bm{\zeta}=\Psi-I_\t\Psi+I_t\Psi-I_\t\bmR_h\Psi$. Then, by the stability of $I_\t$ \eqref{eq:lagrangestability} we get
	\begin{equation}
		\|\Psi-I_\t\Psi\|_{L^{m}(I_n;\bm{L}^{2})}\leq c\t_n^{k+1}\|\p_t^{k+1}\Psi\|_{L^{m}(I_n;\bm{L}^{2})}.\label{eq:lem44_1}
	\end{equation}
		By \eqref{eq:stabilitya} we obtain
		\begin{equation}\label{eq:lem44_2}
			\begin{aligned}
				\|I_\t(\Psi-\bmR_h\Psi)\|_{L^{m}(I_n;\bm{L}^{2})}&\leq c  \|(\Psi-\bmR_h\Psi)\|_{L^{m}(I_n;\bm{L}^{2})}+c\t_n\|\p_t\Psi-\bmR_h\p_t\Psi)\|_{L^{m}(I_n;\bm{L}^{2})}\\
				&\overset{\eqref{eq:ellipticprojectionerrorVector}}{\leq} ch^{r+1}\|\Psi\|_{L^{m}(I_n;\bm{H}^{r+1})}+ c\t_nh^{r+1}\|\p_t\Psi\|_{L^{m}(I_n;\bm{H}^{r+1})}
			\end{aligned}
		\end{equation}
		Using the triangle inequality on $\bm{\zeta}$ with the estimates \eqref{eq:lem44_1} and \eqref{eq:lem44_2} results in \eqref{eq:splitting4}. The inequality \eqref{eq:splitting3} can be shown in a similar fashion using \eqref{eq:ellipticprojectionerrorScalar}.
	\end{proof}
	\subsection{Variational equations for \texorpdfstring{$\bmE_{\t,h}^{1},\bmEth^{2},\bmE_{\t,h}^{3},$}{TEXT} and \texorpdfstring{$e_{\t,h}$}{TEXT}}
	First, we derive the variational equations for the discrete errors. Note that the blue and red terms are important when choosing a test function later to control the errors.
	\begin{lemma}[Variational equations for $\bm{E}_{\t,h}^{1},\bmEth^{2}, e_{\t,h}$ and $\bm{E}_{\t,h}^{3}$]\label{lemma:varerror}
		Let 
		\begin{align}\label{eq:Trelations}
			T_{I}^{n}:=I_\t\int\p_t\bmu-I_\t\p_t\bmu ds,\, \bm{T}_{II}^{n}:=I_\t\bmu-\bmu, \bm{T}_{\bmF_1}^{n}:=\rho\bmF-I_\t(\rho\bmF),\bm{T}^{n}_{\bmF_2}:=\rho_f\bmF-I_\t(\rho_f\bmF).
		\end{align}
		The errors satisfy the equations
		
		\begin{subequations}\label{eq:errorvariational}
			\begin{align}
				&\int_{I_n}\langle \p_t\bmE^{1}_{\t,h},\bmPth^{1}\rangle-\langle  \bmE^{2}_{\t,h},\bmPth^{1}\rangle dt= 0,\label{eq:errorVelocity}
				\\
				\begin{split}
					&\int_{I_n}\langle \p_t\bmE^{2}_{\t,h},\bmPth^{2}\rangle+\langle \bm{A}_{h}\bmE^{1}_{\t,h},\bmPth^{2}\rangle-\underbrace{\a\langle e_{\t,h},\div\bmPth^{2}\rangle}_{\mathfrak{A}_{\a}^{V}} +\rho_f\langle \p_t\bmE^{3}_{\t,h},\bmPth^{2}\rangle dt \label{eq:errorMechanincs}
					\\
					&\quad\quad= \int_{I_n}\langle T_{\bmF_1},\bmPh^{2}\rangle dt
					-\int_{I_n}\langle \p_t\bm{\eta}_{2},\bmPth^{2}\rangle dt-\int_{n}\left\langle \bm{A}_h \bm{T}_{I}^{n},\bmPth^{2}\right\rangle dt
					\\
					&\quad\quad\,+\int_{I_n}\langle\bm{A}_h\bm{T}_{II}^{n},\bmPth^{2}\rangle dt+\int_{I_n}\a\langle \om,\div\bmPth^{2}\rangle dt -\int_{I_n}\rho_f\langle \p_t\bm{\zeta},\bmPth^{2}\rangle dt,
				\end{split}
				\\
				\begin{split}
					&\int_{I_n}c_0\langle\p_te_{\t,h},\psi_{\t,h}\rangle+\underbrace{\a\langle \div\p_t\bmE^{1}_{\t,h},\psi_{\t,h}\rangle}_{\mathfrak{A}_{\a}^{P}}+\underbrace{\langle\div\bmEth^{3},\psi_{\t,h}\rangle}_{\mathfrak{B}_{\eta_{k}}^{P}} dt\label{eq:errorFlow}
					\\
					&\quad\quad=-\int_{I_n}c_0\langle\p_t\om,\psi_{\t,h}\rangle dt-\int_{I_n}\a\langle \div\p_t\bm{\eta}_{1},\psi_{\t,h}\rangle dt -\int_{I_n}\langle\div\bm{\zeta},\psi_{\t,h}\rangle dt,
				\end{split}
				\\
				\begin{split}
					&\int_{I_n}\rho_fc_1F\langle \p_t\bmE^{3}_{\t,h},\bmPth^{3}\rangle+\rho_fF\langle \bmE^{3}_{\t,h},\bmPth^{3}\rangle-\overbrace{d_1\eta_k\langle e_{\t,h},\div\bmPth^{3}\rangle}^{\mathfrak{B}_{\eta_{k}}^{\Psi}}dt
					\\
					&\quad\quad+\int_{I_n}\rho_fd_1\eta_k\langle \p_t\bmE^{2}_{\t,h},\bmPth^{3}\rangle dt \label{eq:errorADE}
					=\int_{I_n}\langle T_{\bmF_2},\bmPth^{3}\rangle dt
					-\int_{I_n}\rho_fc_1F\langle \p_t\bm{\zeta},\bmPth^{3}\rangle dt
					\\
					&\quad\quad-\int_{I_n}\rho_fF\langle \bm{\zeta},\bmPth^{3}\rangle dt+\int_{I_n}d_1\eta_k\langle \om,\div\bmPth^{3}\rangle dt
					-\int_{I_n}\rho_fd_1\eta_k\langle \p_t\bm{\eta}_{2},\bmPth^{3}\rangle dt,\end{split}
			\end{align}
		\end{subequations}
		for all $\bmPth^{1}\in(\P_{k-1}(I_n;V_h^{r})^{d},\bmPth^{2}\in(\P_{k-1}(I_n;V_h^{r})^{d}$, $\psi_{\t,h}\in\P_{\t,h}(I_n;V_h^{r})$ and $\bmPth^{3}\in(\P_{k-1}(I_n;V_h^{r}))^{d}$.
	\end{lemma}
	\begin{proof}
		To derive the variational form in terms of the discrete errors, we first subtract \eqref{variational} from the weak form of \eqref{eqs:firstorderADE} and 
		split the errors using \eqref{eqs:errorsplit} to obtain
		\begin{subequations}
			\begin{align}
				&\int_{I_n}\langle \p_t\bmE^{1}_{\t,h},\bmPth^{1}\rangle-\langle  \bmE^{2}_{\t,h},\bmPth^{1}\rangle dt= -\int_{I_n}\langle \p_t\bm{\eta}_{1},\bmPth^{1}\rangle-\langle  \bm{\eta}_{2},\bmPth^{1}\rangle dt,\label{eq:419a}
				\\
				\begin{split}
					&\int_{I_n}\langle \p_t\bmE^{2}_{\t,h},\bmPth^{2}\rangle+\langle \bm{A}_{h}\bmE^{1}_{\t,h},\bmPth^{2}\rangle-\a\langle e_{\t,h},\div\bmPth^{2}\rangle +\rho_f\langle \p_t\bmE^{3}_{\t,h},\bmPth^{2}\rangle dt 
					\\
					&\quad= \int_{I_n}\langle \bmF-I_{\t}\bmF,\bmPh^{2}\rangle dt
					-\int_{I_n}\langle \p_t\bm{\eta}_{2},\bmPth^{2}\rangle+\langle \bm{A}_{h}\bm{\eta}_{1},\bmPth^{2}\rangle dt\\
					&\quad\,\,-\int_{I_n}\a\langle \om,\div\bmPth^{2}\rangle +\rho_f\langle \p_t\bm{\zeta},\bmPth^{2}\rangle dt,\label{eq:419b}
				\end{split}
				\\
				\begin{split}
					&\int_{I_n}c_0\langle\p_te_{\t,h},\psi_{\t,h}\rangle+\a\langle \div\p_t\bmE^{1}_{\t,h},\psi_{\t,h}\rangle+\langle\div\bmEth^{3},\psi_{\t,h}\rangle dt
					\\
					&\quad=-\int_{I_n}c_0\langle\p_t\om,\psi_{\t,h}\rangle+\a\langle \div\p_t\bm{\eta}_{1},\psi_{\t,h}\rangle+\langle\div\bm{\zeta},\psi_{\t,h}\rangle dt,
				\end{split}
				\\
				\begin{split}
					&\int_{I_n}\rho_fc_1F\langle \p_t\bmE^{3}_{\t,h},\bmPth^{3}\rangle+\rho_fF\langle \bmE^{3}_{\t,h},\bmPth^{3}\rangle-d_1\eta_k\langle e_{\t,h},\div\bmPth^{3}\rangle dt
					\\&\quad+\int_{I_n}\rho_fd_1\eta_k\langle \p_t\bmE^{2}_{\t,h},\bmPth^{3}\rangle dt
					=\int_{I_n}\langle T_{\bmF_2},\bmPth^{3}\rangle dt
					-\int_{I_n}\rho_fc_1F\langle \p_t\bm{\zeta},\bmPth^{3}\rangle dt\\
					&\quad\,\,+\int_{I_n}\rho_fF\langle \bm{\zeta},\bmPth^{3}\rangle-d_1\eta_k\langle \om,\div\bmPth^{3}\rangle
					+\rho_fd_1\eta_k\langle \p_t\bm{\eta}_{2},\bmPth^{3}\rangle dt.
				\end{split}
			\end{align}
		\end{subequations}
		The right-hand side of \eqref{eq:419a} also appears in \cite[Eq. (4.15)]{bause2024} and is equal to zero. We also see that the term with $\bm{A}_h$ on the right-hand side of \eqref{eq:419b} can be decomposed similarly to \cite[Eq. (4.18)]{bause2024},  thus we have
		\begin{align*}
			\langle \bm{A}_{h}\bm{\eta}_{1},\bmPth^{2}\rangle
			=
			\int_{n}\left\langle \bm{A}_h \bm{T}_{I}^{n},\bmPth^{2}\right\rangle dt-\int_{I_n}\langle\bm{A}_h\bm{T}_{II}^{n},\bmPth^{2}\rangle dt,
		\end{align*}
		where $\bm{T}_{I}^{n}$ and $\bm{T}_{II}^{n}$ is defined in \eqref{eq:Trelations}.
	\end{proof}

In the variational error equation \eqref{eq:errorvariational}, we have multiple problematic terms which are difficult to control and require special care in the following analysis. The first two are the terms $\mathfrak{A}_{a}^{P}$ and $\mathfrak{A}_{a}^{V}$ appearing from the coupling terms $\a\del p$ and $\a \div\bmv$, however by choosing test functions appropriately the error contributions from these terms equilibrate each other. Secondly, the terms $\mathfrak{B}_{\eta_{k}}^{P}$ and $\mathfrak{B}_{\eta_{k}}^{\Psi}$  arising from the coupling between the pressure and the auxiliary variable can also be absorbed into each other by a clever choice of test functions. The last two terms which are problematic, are the last terms of the left-hand side of \eqref{eq:errorMechanincs} and \eqref{eq:errorADE}, i.e. $\rho_f\langle \p_t\bmE^{3}_{\t,h},\bmPth^{2}\rangle $ and  $\rho_fd_1\eta_k\langle \p_t\bmE^{2}_{\t,h},\bmPth^{3}\rangle$.  To control these terms, we have chosen to estimate the velocity and ADE term simultaneously to control these terms, similarly to the stability estimate \eqref{eq:errorstabilityestimate}.
We note that a similar idea was used for the dynamic Biot system in \cite{kraus2024}, which uses $\bm{H}$(div) conforming approximation of discontinuous Galerkin type for the displacement and a mixed finite element approach for the flux and pressure.
First recall the matrix $\bm{N}$ defined in \eqref{eq:Nmatrix}, then by summing the error equations \eqref{eq:errorMechanincs}, \eqref{eq:errorFlow} and \eqref{eq:errorADE} we obtain
\begin{equation}\label{eq:errorequationsadded}
	\begin{aligned}
		\int_{I_n}\llangle[\Big] \bm{N}\begin{pmatrix}
			\p_t\bmE^{2}_{\t,h}\\
			\p_t\bmE^{3}_{\t,h}
		\end{pmatrix},\begin{pmatrix}
			\bmPth^{2}\\
			\bmPth^{3}
		\end{pmatrix}\rrangle[\Big]dt+\int_{I_n}\langle \bm{A}_{h}\bmE^{1}_{\t,h},\bmPth^{2}\rangle dt-\int_{I_n}\a\langle e_{\t,h},\div\bmPth^{2}\rangle dt 
		\\+\int_{I_n}c_0\langle\p_te_{\t,h},\psi_{\t,h}\rangle dt
		+\int_{I_n}\a\langle \div\p_t\bmE^{1}_{\t,h},\psi_{\t,h}\rangle dt+\int_{I_n}\langle\div\bmEth^{3},\psi_{\t,h}\rangle dt
		\\+ \int_{I_n}\rho_fF\langle \bmE^{3}_{\t,h},\bmPth^{3}\rangle dt
		- \int_{I_n} d_1\eta_k\langle e_{\t,h},\div\bmPth^{3}\rangle dt
		= \int_{I_n}\llangle[\Big] \begin{pmatrix}
			T_{\bmF_1}\\
			T_{\bmF_2}
		\end{pmatrix},\begin{pmatrix}
			\bmPth^{2}\\
			\bmPth^{3}
		\end{pmatrix}\rrangle[\Big] dt
		\\
		-	\int_{I_n}\llangle[\Big] \bm{N}\begin{pmatrix}
			\p_t\bm{\eta}_{2}\\
			\p_t\bm{\zeta}
		\end{pmatrix},\begin{pmatrix}
			\bmPth^{2}\\
			\bmPth^{3}
		\end{pmatrix}\rrangle[\Big]dt
		-\int_{n}\left\langle \bm{A}_h \bm{T}_{I}^{n},\bmPth^{2}\right\rangle dt
		+\int_{I_n}\langle\bm{A}_h\bm{T}_{II}^{n},\bmPth^{2}\rangle dt
		\\
		+\int_{I_n}\a\langle \om,\div\bmPth^{2}\rangle dt 
		-\int_{I_n}c_0\langle\p_t\om,\psi_{\t,h}\rangle dt-\int_{I_n}\a\langle \div\p_t\bm{\eta}_{1},\psi_{\t,h}\rangle dt
		\\
		-\int_{I_n}\langle\div\bm{\zeta},\psi_{\t,h}\rangle dt
		-\int_{I_n}\rho_fF\langle \bm{\zeta},\bmPth^{3}\rangle dt
		+\int_{I_n}d_1\eta_k\langle \om,\div\bmPth^{3}\rangle dt.
	\end{aligned}
\end{equation}
\subsection{Estimation of interpolation errors}

We then need to estimate the interpolation errors of the right-hand side of \eqref{eq:errorequationsadded}.  Note that if a bold test function is missing an index, a calculation is needed to obtain the term.

\begin{lemma}\label{lemma:estimation}
	[Estimation of interpolation errors]
	Let $k,r\geq 1$. For $T_{I},T_{II}, T_{\bmF_1},T_{\bmF_2}$ defined in \eqref{eq:Trelations} and for  $\bm{\eta}_1,\bm{\eta}_2, \om$ and $\bm{\zeta}$ defined by the error splitting \eqref{eqs:errorsplit} it holds that
	\begin{subequations}
		\begin{align}
			\left|\int_{n}\left\langle \bm{A}_h \bm{T}_{I}^{n},\bmPth^{2}\right\rangle dt\right|&\leq c\t_n^{k+2}\|\del\p_t^{k+2}\bmu\|_{L^{2}(I_{n};\bm{L}^{2})}\norm{\bmPth^{2}}_{L^{2}(I_{n};\bm{L}^{2})},\label{eq:estimation1}
			\\
			\left|\int_{n}\left\langle \bm{A}_h \bm{T}_{II}^{n},\bmPth^{2}\right\rangle dt\right|&\leq c\t_n^{k+1}\|\del\p_t^{k+1}\bmu\|_{L^{2}(I_{n};\bm{L}^{2})}\norm{\bmPth^{2}}_{L^{2}(I_{n};\bm{L}^{2})},\label{eq:estimation2}
			\\
			\left|\int_{n}\left\langle \bm{A}_h \p_t\bm{T}_{II}^{n},\bmPth\right\rangle dt\right|&\leq c\t_n^{k+1}\|\del\p_t^{k+2}\bmu\|_{L^{2}(I_{n};\bm{L}^{2})}\norm{\bmPth}_{L^{2}(I_{n};\bm{L}^{2})},\label{eq:estimation3}
			\\
			\begin{split}
				\left|\int_{I_n}\left\langle \div\p_t\bm{\eta}_1,\psi_{\t,h}\right\rangle dt\right|&\leq c\left(\t_n^{k+1}\|\p_t^{k+2}\bmu\|_{L^{2}(I_{n};\bm{H}^{1})}+h^{r}\|\p_t\bmu\|_{L^{2}(I_{n};\bm{H}^{r+1})}
				\right)\\
				&\,\quad\cdot\left(\|\psi_{\t,h}\|_{L^{2}(I_{n};L^{2})}\right),\label{eq:estimation4}
			\end{split}
			\\
			\begin{split}
				\left|\int_{I_n}\langle \p_t\om,\psi_{\t,h}\rangle dt\right|&\leq c\left(\t_n^{k+1}\|\p_t^{k+2}p\|_{L^{2}(I_{n};L^{2})}+h^{r+1}\|\p_tp\|_{L^{2}(I_{n};H^{r+1})}
				\right)\\
				&\,\quad\cdot\left(\|\psi_{\t,h}\|_{L^{2}(I_{n};L^{2})}\right),\label{eq:estimation5}
			\end{split}
			\\
			\begin{split}
				\left|\int_{I_n}\langle \del\om,\bmPth^{3}\rangle dt\right|&\leq c\Big(\t_{n}^{k+1}\|\p_{t}^{k+1}p\|_{L^{2}(I_{n};H^{1})}+h^{r}\|p\|_{L^{2}(I_{n};H^{r+1})}\\
				&\quad\quad+\t_n h^{r} \|\p_tp\|_{L^{2}(I_{n};H^{r+1})}\Big)
				\left\|\bmPth^{3}\right\|_{L^{2}(I_n;\bm{L}^{2})},\label{eq:estimation8}
			\end{split}
			\\
			\begin{split}
				\left|\int_{I_n}\langle \div\bm{\zeta},\psi_{\t,h}\rangle dt\right|&\leq c \Big(\t_{n}^{k+1}\|\p_t^{k+1}\Psi\|_{L^{2}(I_{n};\bm{H}^{1})}+h^{r}\left\|  \Psi \right\|_{L^{2}(I_n;\bm{H}^{r+1})}\\
				&\quad\quad+\t_n h^{r}\left\|  \p_t\Psi \right\|_{L^{2}(I_n;\bm{H}^{r+1})}\Big)\|\psi_{\t,h}\|_{L^{2}(I_{n};L^{2})},\label{eq:estimationofDivZeta}
			\end{split}
			\\
			\begin{split}
				\left|\int_{I_n}\langle \bm{\zeta},\bmPth^{3}\rangle dt\right|&\leq c\Big(\t_{n}^{k+1}\|\p_t^{k+1}\Psi\|_{L^{2}(I_{n};\bm{L}^{2})}+h^{r+1}\left\|  \Psi \right\|_{L^{2}(I_n;\bm{H}^{r+1})}\\
				&\quad\quad+\t_nh^{r+1}\left\|  \p_t\Psi \right\|_{L^{2}(I_n;\bm{H}^{r+1})}\Big)\|\bmPth^{3}\|_{L^{2}(I_{n};\bm{L}^{2})},\label{eq:estimationofZeta}
			\end{split}
			\\
			\begin{split}
				\int_{I_n}\llangle[\Big] \bm{N}\begin{pmatrix}
					\p_t\bm{\eta}_{2}\\
					\p_t\bm{\zeta}
				\end{pmatrix},\begin{pmatrix}
					\bmPth^{2}\\
					\bmPth^{3}
				\end{pmatrix}\rrangle[\Big]dt
				&\leq c\Big(\t_{n}^{k+1}\left(\|\p_t^{k+2}\Psi\|_{L^{2}(I_{n};\bm{L}^{2})}+\|\p_t^{k+2}\bmv\|_{L^{2}(I_{n};\bm{L}^{2})}\right)
				\\
				&\quad+h^{r+1}\left(\|\p_t\Psi\|_{L^{2}(I_{n};\bm{H}^{r+1})}+\|\p_t\bmv\|_{L^{2}(I_{n};\bm{H}^{r+1})}\right)\Big)
				\\
				&\quad\cdot\Big(\|\bm{N}^{\frac12}\begin{pmatrix}
					\bmPth^{2}\\
					\bmPth^{3}
				\end{pmatrix}\|_{L^{2}(I_{n};\bm{L}^{2})}\Big),\label{eq:estimationmatrix}
			\end{split}
			\\
			\,\int_{I_n}\llangle[\Big] \begin{pmatrix}
				T_{\bmF_1}\\
				T_{\bmF_2}
			\end{pmatrix},\begin{pmatrix}
				\bmPth^{2}\\
				\bmPth^{3}
			\end{pmatrix}\rrangle[\Big] dt&\leq c\t_n^{k+1}\|\begin{pmatrix}
				\rho\p_t^{k+1}\bmF\\\rho_f\p_t^{k+1}\bmF
			\end{pmatrix}\|_{L^{2}(I_n;\bm{L}^{2})}\|\begin{pmatrix}
				\bmPth^{2}\\
				\bmPth^{3}
			\end{pmatrix}\|_{L^{2}(I_{n};\bm{L}^{2})},\label{eq:estimationf}
		\end{align}
	\end{subequations}
\end{lemma}
\begin{proof}
	The first three inequalities \eqref{eq:estimation1} to \eqref{eq:estimation3} can be shown similarly to \cite[Lem. 3.3]{makridakis2005} while also using the norm equivalence \eqref{eq:norm}. Do note that for \eqref{eq:estimation1} the order is of $\t_n^{k+2}$ because the $H^{1}-L^{2}$ inverse inequality is not used, and therefore the derivative order is also higher. The next two inequalities, \eqref{eq:estimation4} and \eqref{eq:estimation5}, can be shown along similar lines as \cite[Lem. 4.3, Eqs. (4.21d) and (4.21e)]{bause2024}.  First, we prove \eqref{eq:estimation8}. Using the definition of $\om$ from \eqref{eq:pressureerrorsplit} while adding and subtracting $I_{\t}p$ we get
	\begin{align}\label{eq:loss}
		\int_{I_n}\langle \del\om,\bmPth^{3}\rangle dt= \underbrace{\int_{I_n}\langle \del(p-I_{\t}p),\bmPth^{3}\rangle dt}_{\Gamma_{1}}+\underbrace{\int_{I_n}\langle \del I_{\t}(p-R_h p),\bmPth^{3}\rangle dt}_{\Gamma_{2}}.
	\end{align}
	First we consider $\Gamma_{1}$. By using that the Lagrange interpolation satisfies \eqref{eq:lagrangestability}, along with the Cauchy-Schwarz inequality and \eqref{eq:ellipticprojectionerrorScalar}, we obtain
	\begin{equation*}
		\begin{aligned}
			|\Gamma_{1}|\leq& \,\left|\int_{I_n}\langle \del(p-I_{\t}p),\bmPth^{3}\rangle dt\right|\leq \|(p-I_{\t}p)\|_{L^{2}(I_{n};H^{1})}\|\bmPth^{3}\|_{L^{2}(I_{n};\bm{L}^{2})}
			\\
			\leq& \,c\t_{n}^{k+1}\|\p_t^{k+1}p\|_{L^{2}(I_n;H^{1})}\|\bmPth^{3}\|_{L^{2}(I_{n};\bm{L}^{2})}.
		\end{aligned}
	\end{equation*}
	
			Next we consider $\Gamma_{2}$ in \eqref{eq:loss}. To ease notation, we let $\xi=p-R_hp$. 
			By the stability result \eqref{eq:stabilitya} and the error estimate for the elliptic projection operator \eqref{eq:ellipticprojectionerrorScalar}, we obtain the inequality
			\begin{equation*}
				\begin{aligned}
					|\Gamma_{2}|&=\left|\int_{I_n}\langle \del I_{\t}\xi,\bmPth^{3}\rangle dt\right|\leq \|  \del I_{\t}\xi \|_{L^{2}(I_n;L^{2})}\left\|\bmPth^{3}\right\|_{L^{2}(I_n;\bm{L}^{2})}\\
					&\overset{\eqref{eq:stabilitya}}{\leq}c\left(\|\del\xi\|_{L^{2}(I_n;L^{2})}+\t_n\|\del\p_t\xi\|_{L^{2}(I_n;L^{2})}\right)\left\|\bmPth^{3}\right\|_{L^{2}(I_n;\bm{L}^{2})}
					\\
					&\leq\,ch^{r}(\left\|  p \right\|_{L^{2}(I_n;H^{r+1})}+\t_n\left\|  \p_tp \right\|_{L^{2}(I_n;H^{r+1})}\left\|\bmPth^{3}\right\|_{L^{2}(I_n;\bm{L}^{2})},
				\end{aligned}
			\end{equation*}
			from which we can infer \eqref{eq:estimation8}.The proof of \eqref{eq:estimationofDivZeta} follows analogously using \eqref{eq:ellipticprojectionerrorVector}. \eqref{eq:estimationofZeta} is a direct estimation using similar arguments.
			Lastly, we consider \eqref{eq:estimationmatrix}. Note that from \eqref{eqs:errorsplit} we have
			\begin{equation*}
				\bm{\zeta}=\Psi- I_{\t}\Psi +I_{\t}\Psi-I_{\t}\bm{R}_{h}\Psi\mbox{ and }\bm{\eta}_{2}=\bmv-I_{\t}\bmv+I_{\t}\bmv-\bmw_{2}.
			\end{equation*}
			Using this gives
			\begin{align*}
				\int_{I_n}\llangle[\Big] \bm{N}\begin{pmatrix}
					\p_t\bm{\eta}_{2}\\
					\p_t\bm{\zeta}
				\end{pmatrix},\begin{pmatrix}
					\bmPth^{2}\\
					\bmPth^{3}
				\end{pmatrix}\rrangle[\Big]dt
				&=	\underbrace{\int_{I_n}\llangle[\Big] \bm{N}\begin{pmatrix}
						\p_t(\bmv-I_{\t}\bmv)\\
						\p_t(\Psi- I_{\t}\Psi)
					\end{pmatrix},\begin{pmatrix}
						\bmPth^{2}\\
						\bmPth^{3}
					\end{pmatrix}\rrangle[\Big]dt}_{:=\Gamma_{3}}
				\\&\quad+	
				\underbrace{\int_{I_n}\llangle[\Big] \bm{N}\begin{pmatrix}
						\p_t(I_{\t}\bmv-\bmw_{2} )\\
						\p_t(I_{\t}\Psi-I_{\t}\bm{R}_{h}\Psi)
					\end{pmatrix},\begin{pmatrix}
						\bmPth^{2}\\
						\bmPth^{3}
					\end{pmatrix}\rrangle[\Big]dt}_{:=\Gamma_4}.
			\end{align*}
			Let us first consider $\Gamma_{3}$ for $k\geq 2$. Recall that the endpoints of $I_n$ are included in the set of Gauss-Lobatto quadrature points of $I_n$. Therefore, by applying integration by parts in time, we obtain
			\begin{align*}
				\Gamma_{3}=\int_{I_n}\llangle[\Big] \bm{N}\begin{pmatrix}
					\p_t(\bmv-I_{\t}\bmv)\\
					\p_t(\Psi- I_{\t}\Psi)
				\end{pmatrix},\begin{pmatrix}
					\bmPth^{2}\\
					\bmPth^{3}
				\end{pmatrix}\rrangle[\Big]dt =- \int_{I_n}\llangle[\Big] \bm{N}\begin{pmatrix}
					(\bmv-I_{\t}\bmv)\\
					(\Psi- I_{\t}\Psi)
				\end{pmatrix},\begin{pmatrix}
					\p_t\bmPth^{2}\\
					\p_t\bmPth^{3}
				\end{pmatrix}\rrangle[\Big]dt.
			\end{align*}
			Now we denote the Lagrange interpolation operator at $k+2$ points of the interval $\widebar{I}_{n}=[t_{n-1},t_{n}]$ which consists the $k+1$ Gauss-Lobatto quadrature nodes $t_{n,\mu}$, for $\mu=0,...,k$ and a further node in $(t_{n-1},t_{n})$ that is distinct from the previous ones by $I_{n}^{k+1}$. Then, $(I_{\t}^{k+1}\bmv)\p_{t}\bmPth$ and $(I_{\t}^{k+1}\Psi)\p_{t}\bmPth$ is polynomials of degree $2k-1$ in $t$ such that
			\begin{align*}
				\int_{I_n}\llangle[\Big] \bm{N}\begin{pmatrix}
					(\bmv-I_{\t}\bmv)\\
					(\Psi- I_{\t}\Psi)
				\end{pmatrix},\begin{pmatrix}
					\p_t\bmPth^{2}\\
					\p_t\bmPth^{3}
				\end{pmatrix}\rrangle[\Big]dt=\int_{I_n}\llangle[\Big] \bm{N}\begin{pmatrix}
					(\bmv-I_{\t}^{k+1}\bmv)\\
					(\Psi- I_{\t}^{k+1}\Psi)
				\end{pmatrix},\begin{pmatrix}
					\p_t\bmPth^{2}\\
					\p_t\bmPth^{3}
				\end{pmatrix}\rrangle[\Big]dt.
			\end{align*}
			Further, by integration by parts and the stability of the operator $I_{\t}^{k+1}$ in the norm of $L^{2}(I_n;L^{2})$ and the symmetric positive definiteness of $\bm{N}$, we get
			\begin{equation}
				\begin{aligned}
					|\Gamma_{3}|&\leq \int_{I_n}\llangle[\Big] \bm{N}\begin{pmatrix}
						\p_t(\bmv-I_{\t}^{k+1}\bmv)\\
						\p_t(\Psi- I_{\t}^{k+1}\Psi)
					\end{pmatrix},\begin{pmatrix}
						\bmPth^{2}\\
						\bmPth^{3}
					\end{pmatrix}\rrangle[\Big]dt\\
					&\leq c\t_{n}^{k+1}\left(\|\p_t^{k+2}\Psi\|_{L^{2}(I_{n};\bm{L}^{2})}+\|\p_t^{k+2}\bmv\|_{L^{2}(I_{n};\bm{L}^{2})}\right)\|\bm{N}^{\frac12}\begin{pmatrix}
						\bmPth^{2}\\
						\bmPth^{3}
					\end{pmatrix}\|_{L^{2}(I_{n};\bm{L}^{2})}.
				\end{aligned}
			\end{equation}
			For $k=1$, we have that $\p_tI_\t\bmv, \p_tI_\t\Psi\,\in \mathbb{P}_0(I_n;V^{r}_h)$ where $\p_tI_\t\bmv=(\bmv(t_n)-\bmv(t_{n-1}))/\t_n$ and $\p_tI_\t\Psi=(\Psi(t_n)-\Psi(t_{n-1}))/\t_n$. Consequently, it follows that $\Gamma_3=0$. Next, we need to estimate $\Gamma_{4}$. Recall from \eqref{eq:special approximation} that $\bmw_2=I_{\t}(\bm{R}_h\bmv)$.	Let $\bm{\xi}_2=\bmv-\bmR_h\bmv$ and $\bm{\xi}_3=\Psi-\bmR_h\Psi$
			then by viewing $\bm{\xi}_i(t_{n-1}^{+})$ as a function constant in time for $i\in\{2,3\}$ we get
			\begin{equation*}
				\begin{aligned}
					|\Gamma_{4}|=&\,\left|\int_{I_n}\llangle[\Big] \bm{N}\begin{pmatrix}
						\p_t(I_{\t}\bm{\xi}_2 )\\
						\p_t(I_{\t}\bm{\xi}_3)
					\end{pmatrix},\begin{pmatrix}
						\bmPth^{2}\\
						\bmPth^{3}
					\end{pmatrix}\rrangle[\Big]dt\right|=\left|\int_{I_n}\llangle[\Big] \bm{N}\begin{pmatrix}
						\p_t(I_{\t}(\bm{\xi}_2-\bm{\xi}_2(t_{n-1}^{+})) )\\
						\p_t(I_{\t}(\bm{\xi}_3-\bm{\xi}_3(t_{n-1}^{+})))
					\end{pmatrix},\begin{pmatrix}
						\bmPth^{2}\\
						\bmPth^{3}
					\end{pmatrix}\rrangle[\Big]dt\right|
					\\
					=&\,\left|\int_{I_n}\llangle[\Big] \bm{N}\begin{pmatrix}
						\p_t I_{\t}\int_{t_{n-1}}^{t}\bm{\xi}_2 ds\\
						\p_t I_{\t}\int_{t_{n-1}}^{t}\bm{\xi}_3 ds
					\end{pmatrix},\begin{pmatrix}
						\bmPth^{2}\\
						\bmPth^{3}
					\end{pmatrix}\rrangle[\Big]dt\right|.
				\end{aligned}
			\end{equation*} 
			Then, we obtain the inequality 
			\begin{equation*}
				\begin{aligned}
					|\Gamma_4|\leq  c\left(\|\p_t I_{\t}\int_{t_{n-1}}^{t}\bm{\xi}_2 ds\|_{L^{2}(I_{n};\bm{L}^{2})}+\|\p_t I_{\t}\int_{t_{n-1}}^{t}\bm{\xi}_3 ds\|_{L^{2}(I_{n};\bm{L}^{2})}\right)\|\bm{N}^{\frac12}\begin{pmatrix}
						\bmPth^{2}\\
						\bmPth^{3}
					\end{pmatrix}\|_{L^{2}(I_{n};\bm{L}^{2})}.
				\end{aligned}
			\end{equation*}
			Both terms can be estimated similarly, and therefore, we only show one.
			First, we apply the  
			$H^{1}-L^{2}$ inverse inequality \eqref{eq:HoneLtwoInverseIneq} along with the stability results \eqref{eqs:stability}. Then we recall that $\bm{\xi}_2(t_{n-1}^{+})$ is a constant function in time, and lastly we use the error estimate for the elliptic projection operator \eqref{eq:ellipticprojectionerrorVector} to obtain the inequality
			\begin{align*}
				\|\p_t I_{\t}\int_{t_{n-1}}^{t}\bm{\xi}_2 ds\|_{L^{2}(I_{n};\bm{L}^{2})}&\overset{\eqref{eq:HoneLtwoInverseIneq}}{\leq} c\t_{n}^{-1}\left\|  I_{\t}\int_{t_{n-1}}^{t}\p_t\bm{\xi}_2 ds\right\|_{L^{2}(I_n;\bm{L}^{2})}
				\\
				&\overset{\eqref{eq:stabilitya}}{\leq}c\left(\t_{n}^{-1}\left\|  \int_{t_{n-1}}^{t}\p_t\bm{\xi}_2 ds\right\|_{L^{2}(I_n;\bm{L}^{2})}+\left\|  \p_t\int_{t_{n-1}}^{t}\p_t\bm{\xi}_2 ds\right\|_{L^{2}(I_n;\bm{L}^{2})}\right)
				\\
				&\overset{\eqref{eq:stabilityb}}{\leq}\,c\left\|  \p_t\bm{\xi}_2 \right\|_{L^{2}(I_n;\bm{L}^{2})}
				\\&\overset{\eqref{eq:ellipticprojectionerrorVector}}{\leq}
				ch^{r+1}\left\|  \p_t\bmv \right\|_{L^{2}(I_n;\bm{H}^{r+1})}.
			\end{align*}
			From these inequalities, we can infer \eqref{eq:estimationmatrix}. The estimate \eqref{eq:estimationf} follows from \eqref{eq:lagrangestability}.
			
		\end{proof}
		
		\begin{remark}
			In \cite{bauseoptimal2024}, the approximation of the dynamic Biot model without memory effects is investigated. An equal-order in space finite element approach is considered. For this model, the estimate \eqref{eq:estimation4} can be sharpened. We note that in our context, this improvement is not possible due to the lack of a second-order space diffusion term in the pressure equation. In the error estimation, its presence would enable the absorption of suboptimal order contributions to the upper error bound. 
		\end{remark}

		\subsection{Estimation of \texorpdfstring{$\bmE_{\t,h}^{1},\bmE_{\t,h}^{2},\bmEth^{3},$}{TEXT} and \texorpdfstring{$e_{\t,h}$}{TEXT}}
		Before we show a stability estimate for the errors, we introduce the following notation with \Cref{lemma:estimation} in mind,
		\begin{equation}\label{eq:notationsummation}
			\begin{aligned}
				\varepsilon_{t}^{n,1}&:=\|\del\p_t^{k+2}\bmu\|_{L^{2}(I_{n};\bm{L}^{2})}+\|\p_t^{k+3}\bmu\|_{L^{2}(I_{n};\bm{L}^{2})}+\|\p_t^{k+1}p\|_{L^{2}(I_{n};H^{1})}\\&\quad+\|\p_t^{k+1}\Psi\|_{L^{2}(I_{n};\bm{H}^{1})}+\|\p_t^{k+2}\bmu\|_{L^{2}(I_{n};\bm{H}^{1})}
				+\|\p_t^{k+2}p\|_{L^{2}(I_{n};L^{2})}\\&\quad+\|\p_t^{k+2}\Psi\|_{L^{2}(I_{n};\bm{L}^{2})}+\|\p_t^{k+1}\Psi\|_{L^{2}(I_{n};\bm{L}^{2})}
				+\|\p_t^{k+2}\bmu\|_{L^{2}(I_n;\bm{H}^{1})}\\
				&\quad+\|\p_t^{k+2}\bmu\|_{L^{\infty}(I_n;\bm{H}^{1})}+\|\p_t^{k+1}\bmu\|_{L^{\infty}(I_n;\bm{H}^{1})} +\varepsilon_{\bmF}^{n}
				,\\
				\varepsilon_{x}^{n,0}&:= \|\p_t\bmu\|_{L^{2}(I_{n};\bm{H}^{r+1})}
				+\|p\|_{L^{2}(I_{n};H^{r+1})}
				+\|\p_t\Psi\|_{L^{2}(I_{n};\bm{H}^{r+1})}+h \varepsilon_{x}^{n,1}\\
				&\quad+h\|p\|_{L^{\infty}(I_n;H^{r+1})}+\t_n\left\|  \p_t\Psi \right\|_{L^{2}(I_n;\bm{H}^{r+1})}+\t_n\left\|  \p_t p \right\|_{L^{2}(I_n;\bm{H}^{r+1})},
				\\
				\varepsilon_{x}^{n,1}&:=\|\p_t p\|_{L^{2}(I_{n};H^{r+1})}
				+\|\Psi\|_{L^{2}(I_{n};\bm{H}^{r+1})}
				+\|\p_t\Psi\|_{L^{2}(I_{n};\bm{H}^{r+1})}
				+\|\p_t^{2}\bmu\|_{L^{2}(I_{n};\bm{H}^{r+1})}\\
				&\quad+\t_n\left\|  \p_t\Psi \right\|_{L^{2}(I_n;\bm{H}^{r+1})}, 
				\\
				\varepsilon_{\bmF}^{n}&:=\|\begin{pmatrix}\rho\p_t^{k+1}\bmF\\\rho_f\p_t^{k+1}\bmF\end{pmatrix}\|_{L^{2}(I_{n};\bm{L}^{2})}.
			\end{aligned}
		\end{equation}
		Here,  $\|\p_t^{k+2}\bmu\|_{L^{2}(I_n;\bm{H}^{1})}$, $\|\p_t^{k+2}\bmu\|_{L^{\infty}(I_n;\bm{H}^{1})}$, $\|\p_t^{k+1}\bmu\|_{L^{\infty}(I_n;\bm{H}^{1})}$ and $\|p\|_{L^{\infty}(I_n;H^{r+1})}$ are included in the anticipation of the terms $\varepsilon_{t}^{n-1,2}$, $\varepsilon_{t}^{n-1,3}$, $\varepsilon_{t}^{n-1,4}$ and $\varepsilon_{x}^{n,2}$ which appears in Lemma \ref{lemma:estimation2}, Lemma \ref{lemma:delta1} and Lemma \ref{lemma:newname}.
		\begin{theorem}[Stability estimate]\label{theorem:stability}
			For $n=1,...,N$, where the errors $\bmE_{\t,h}^{1}, \bmE_{\t,h}^{2}, e_{\t,h}$ and $\bmEth^{3}$ is defined by \eqref{eqs:errorsplit}, 
			it holds that
			\begin{equation}\label{eq:stabilityestimate}
				\begin{aligned}
					&\norm{\bmE_{\t,h}^{1}(t_{n})}_e+\frac12\left\|\bm{N}^{\frac12}\begin{pmatrix}
						\bmE^{2}_{\t,h}(t_n)\\\bmEth^{3}(t_n)
					\end{pmatrix}\right\|_{\bm{L}^{2}}^{2}+\frac{c_0}{2}\|e_{\t,h}(t_n)\|_{\bm{L}^{2}}^{2}	+C_{\e}\|\bmEth^{3}\|_{L^{2}(I_n:\bm{L}^{2})}^{2}
					\\
					&
					\leq\norm{\bmE_{\t,h}^{1}(t_{n-1}^{+})}_e+\frac{c_0}{2}\|e_{\t,h}(t_{n-1}^{+})\|_{L^{2}}^{2}+\frac12\left\|\bm{N}^{\frac12}\begin{pmatrix}
						\bmE^{2}_{\t,h}(t_{n-1}^{+})\\\bmEth^{3}(t_{n-1}^{+})
					\end{pmatrix}\right\|_{\bm{L}^{2}}^{2}
					\\
					&\quad+c\norm{\bmE_{\t,h}^{1}}^{2}_{L^{2}(I_n;\bm{L}^{2})}
					+c\|e_{\t,h}\|_{L^{2}(I_{n};L^{2})}^{2}+\e\|\bm{N}^{\frac12}\begin{pmatrix}
						\bmE_{\t,h}^{2}\\
						\bmEth^{3}
					\end{pmatrix}\|_{L^{2}(I_{n};\bm{L}^{2})}^{2}
					\\
					&\quad+c\t_{n}^{2(k+1)}(\varepsilon_{t}^{n,1})^{2}+ch^{2(r)}(\varepsilon_{x}^{n,0})^{2}+\delta_{1,n}-\delta_{1,n-1}^{+}+\delta_{2,n}-\delta_{2,n-1}^{+},
				\end{aligned}
			\end{equation}
			where $\varepsilon_{t}^{n,1}$ and $\varepsilon_{x}^{n,0}$ are defined in \eqref{eq:notationsummation}
			and 
			\begin{subequations}\label{eqs:deltasinTheorem}
				\begin{align}
					\delta_{1,n}&:=\langle A_hT_{II}(t_n),\bmE_{\t,h}^{1}(t_n)\rangle,\quad&&\delta_{1,n-1}^{+}:=\langle A_hT_{II}\left(t_{n-1}^{+}\right),\bmE_{\t,h}^{1}\left(t_{n-1}^{+}\right)\rangle,
					\\
					\delta_{2,n}&:= \a\langle\om(t_n),\div\bmE_{\t,h}^{1}(t_n)\rangle,\quad&&\delta_{2,n-1}^{+}:=\a\langle \om(t_{n-1}^{+}),\div\bmE_{\t,h}^{1}(t_{n-1}^{+})\rangle.
				\end{align}
			\end{subequations}
		\end{theorem}
		\begin{proof}
			The choice of test functions is important for balancing the errors. In particular, we want the terms $\mathfrak{B}_{\eta_{k}}^{P}$ and $\mathfrak{B}_{\eta_{k}}^{\Psi}$ and the terms $\mathfrak{A}_{\a}^{P}$ and $\mathfrak{A}_{\a}^{V}$ to cancel. Therefore in \eqref{eq:errorequationsadded}, we choose the test functions
			\begin{align}\label{eq:testfunction}
				\bmPth^{2}=\Pi_{\t}^{k-1}\bmE_{\t,h}^{2}
				\mbox{ and } \psi_{\t,h}=\Pi_{\t}^{k-1}e_{\t,h}, \mbox{ and } \bmPth^{3}=\Pi_{\t}^{k-1}\frac{1}{d_1\eta_k}\bmEth^{3}.
			\end{align}
			
			First, we consider the terms $\mathfrak{A}_{\a}^{P}$ and $\mathfrak{A}_{\a}^{V}$ in the equation, i.e., the coupling between the pressure and mechanics, with the test function defined in \eqref{eq:testfunction}. We note that this term also appears in \cite[Thm. 4.4.]{bause2024} and cancels due to the choice of test functions \eqref{eq:testfunction}. This is due to the definition of $(\bmE_{\t,h}^{1},\bmE_{\t,h}^{2})$ in \eqref{eqs:errorsplit} along with the relation between $(\bmw_1,\bmw_2)$ in \eqref{eq:lemma:4_2}. Also $(\bmu_{\t,h},\bmv_{\t,h})$ satisfy the first equation in \eqref{variational}, which leads to the observation
			\begin{align}\label{eq:error2topartialerror1}
				\bmE_{\t,h}^{2}(t_{n,\mu}^{G})=\p_t\bmE_{\t,h}^{1}(t_{n,\mu}^{G}),
			\end{align}
			for $\mu=1,...,k$.  
			Also note that by choosing the test functions in \eqref{eq:testfunction}, the terms $\mathfrak{B}_{\eta_{k}}^{P}$ and $\mathfrak{B}_{\eta_{k}}^{\Psi}$ cancel. 
		By the observation \eqref{eq:error2topartialerror1}, using that $\p_t\bmE_{\t,h}^{1}\in(\P_{k-1}(I_n;V_h^{r}))^{d}$ and \eqref{eqs:lemma2_1} we obtain
		\begin{equation*}\begin{aligned}\label{eq:omtoE2toE1}
				\int_{I_n}\langle\bm{A}_h\bm{T}_{II}^{n},\bmE_{\t,h}^{2}\rangle dt &= 
				\int_{I_n}\langle \Pi_{\t}^{k-1}\bm{A}_h\bm{T}_{II}^{n},\div\Pi_{\t}^{k-1}\bmE_{\t,h}^{2}\rangle dt \\&=
				\frac{\t_n}{2}\sum_{\mu=1}^{k}\hat{\om}_{\mu}^{G}\langle \Pi_\t^{k-1}\bm{A}_h\bm{T}_{II}^{n}(t_{n,\mu}^{G}),\div\bmE_{\t,h}^{2}(t_{n,\mu}^{G})\rangle\nonumber\\&=\frac{\t_n}{2}\sum_{\mu=1}^{k}\hat{\om}_{\mu}^{G}\langle \Pi_\t^{k-1}\bm{A}_h\bm{T}_{II}^{n}(t_{n,\mu}^{G}),\div\p_t\bmE_{\t,h}^{1}(t_{n,\mu}^{G})\rangle
				\\&=
				\int_{I_n}\langle \Pi_{\t}^{k-1}\bm{A}_h\bm{T}_{II}^{n},\div\p_t\bmE_{\t,h}^{1}\rangle dt
				=
				\int_{I_n}\langle \bm{A}_h\bm{T}_{II}^{n},\div\p_t\bmE_{\t,h}^{1}\rangle dt.
		\end{aligned}\end{equation*}
		Then by integration by parts, we get
		\begin{equation*}
			\begin{aligned}
				\int_{I_n}\langle\bm{A}_h\bm{T}_{II}^{n},\bmE_{\t,h}^{2}\rangle dt=\int_{I_n}\langle\bm{A}_h\bm{T}_{II}^{n},\p_t\bmE_{\t,h}^{1}\rangle dt=-\int_{I_n}\langle\bm{A}_h\p_t\bm{T}_{II}^{n},\bmE_{\t,h}^{1}\rangle dt+\d_{1,n}-\d_{1,n-1}^{+},
			\end{aligned}
		\end{equation*}
		where $\d_{1,n},\, \d_{1,n-1}^{+}$ are defined in \eqref{eqs:deltasinTheorem}.
		We also have from \cite[Eq. (4.36)-(4.37)]{bause2024} that 
		\begin{align}\label{eq:integrationByPartsE2toE1}
			\int_{I_n}\langle \om,\div\Pi_{\t}^{k-1}\bmE_{\t,h}^{2}\rangle dt 
			&= 
			-\int_{I_n}\langle \p_t\om,\div\bmE_{\t,h}^{1}\rangle dt+\d_{2,n} -\d_{2,n-1}^{+}.
		\end{align}
		Now we return to the first term of \eqref{eq:errorequationsadded} with the test function \eqref{eq:testfunction} and see that
		\begin{equation}
			\begin{aligned}
				\int_{I_n}\llangle[\Big] \bm{N}\begin{pmatrix}
					\p_t\bmE^{2}_{\t,h}\\
					\p_t\bmE^{3}_{\t,h}
				\end{pmatrix},\begin{pmatrix}
					\Pi_{\t}^{k-1}\bmE_{\t,h}^{2}\\
					\Pi_{\t}^{k-1}\bmEth^{3}
				\end{pmatrix}\rrangle[\Big]dt\underset{\eqref{eqs:lemma2_1}}{\overset{\eqref{gaussquadratureInKthpoint}}{=}}
				\frac{\t_n}{2}\sum_{\mu=1}^{k}\hat{\om}_{\mu}^{G}\llangle[\Bigg]\bm{N}\begin{pmatrix}
					\p_t\bmE_{\t,h}^{2}(t_{\t,\mu}^{G})\\
					\p_t\bmEth^{3}(t_{\t,\mu}^{G})
				\end{pmatrix},\begin{pmatrix}
					\bmE_{\t,h}^{2}(t_{\t,\mu}^{G})\\
					\bmEth^{3}(t_{\t,\mu}^{G})
				\end{pmatrix}\rrangle[\Bigg]
				\\
				=\,\int_{I_n}\llangle[\Big] \bm{N}\begin{pmatrix}
					\p_t\bmE^{2}_{\t,h}\\
					\p_t\bmE^{3}_{\t,h}
				\end{pmatrix},\begin{pmatrix}
					\bmE_{\t,h}^{2}\\
					\bmEth^{3}
				\end{pmatrix}\rrangle[\Big]dt
				=\,
				\frac12\left\|\bm{N}^{\frac12}\begin{pmatrix}
					\bmE^{2}_{\t,h}(t_n)\\\bmEth^{3}(t_n)
				\end{pmatrix}\right\|_{\bm{L}^{2}}^{2}-\frac12\left\|\bm{N}^{\frac12}\begin{pmatrix}
					\bmE^{2}_{\t,h}(t_{n-1}^{+})\\\bmEth^{3}(t_{n-1}^{+})
				\end{pmatrix}\right\|_{\bm{L}^{2}}^{2}.
			\end{aligned}
		\end{equation}
		By looking at the second term in \eqref{eq:errorequationsadded} and using \eqref{eqs:lemma2_1} and the symmetry of $\bm{A}_h$ we observe that
		\begin{equation}
			\begin{aligned}
				\int_{I_n}\langle \bm{A}_{h}\bmE^{1}_{\t,h},\Pi_{\t}^{k-1}\bmE_{\t,h}^{2}\rangle dt=&\,\frac{\t_n}{2}\sum_{\mu=1}^{k}\hat{\om}_{\mu}^{G}\langle \bm{A}_{h}\bmE^{1}_{\t,h}(t_{n,\mu}^{G}),\p_t\bmE_{\t,h}^{1}(t_{n,\mu}^{G})\rangle 
				\\
				=&\, \int_{I_n}\langle \bm{A}_{h}\bmE^{1}_{\t,h},\p_t\bmE_{\t,h}^{1}\rangle dt=\,\norm{\bmE_{\t,h}^{1}(t_{n})}_e^{2}-\norm{\bmE_{\t,h}^{1}(t_{n-1}^{+})}_e^{2}.
			\end{aligned}
		\end{equation}
		Note that the last term on the first line of \eqref{eq:errorequationsadded} with the test function defined in \eqref{eq:testfunction} is similar to \cite[Eq. (4.38)]{bause2024}, thus we obtain
		\begin{equation}\label{eq:pressureErrorToNorm}
			\begin{aligned}
				c_0\int_{I_n}\langle \p_t e_{\t,h},\Pi_{\t}^{k-1}e_{\t,h}\rangle dt 
				=\frac{c_0}{2}\|e_{\t,h}(t_n)\|_{L^{2}}^{2}-\frac{c_0}{2}\|e_{\t,h}(t_{n-1}^{+})\|_{L^{2}}^{2}.
			\end{aligned}
		\end{equation}
		Returning to \eqref{eq:errorequationsadded} and using \eqref{eq:testfunction} to \eqref{eq:pressureErrorToNorm} we get
		\begin{equation}
			\begin{aligned}
				&\norm{\bmE_{\t,h}^{1}(t_{n})}_e^{2}+\frac12\left\|\bm{N}^{\frac12}\begin{pmatrix}
					\bmE^{2}_{\t,h}(t_n)\\\bmEth^{3}(t_n)
				\end{pmatrix}\right\|_{\bm{L}^{2}}^{2}+\frac{c_0}{2}\|e_{\t,h}(t_n)\|_{L^{2}}^{2}+\rho_f Fd_1^{-1}\eta_k^{-1}\|\bmEth^{3}\|_{L^{2}(I_n:\bm{L}^{2})}^{2}	
				\\
				&
				=\norm{\bmE_{\t,h}^{1}(t_{n-1}^{+})}_e^{2}+\frac{c_0}{2}\|e_{\t,h}(t_{n-1}^{+})\|_{L^{2}}^{2}+\frac12\left\|\bm{N}^{\frac12}\begin{pmatrix}
					\bmE^{2}_{\t,h}(t_{n-1}^{+})\\\bmEth^{3}(t_{n-1}^{+})
				\end{pmatrix}\right\|_{\bm{L}^{2}}^{2}
				\\
				&\quad+\int_{I_n}\llangle[\Big] \begin{pmatrix}
					T_{\bmF_1}\\
					T_{\bmF_2}
				\end{pmatrix},\begin{pmatrix}
					\Pi_{\t}^{k-1}\bmE_{\t,h}^{2}\\
					d_1^{-1}\eta_k^{-1}\Pi_{\t}^{k-1}\bmEth^{3}
				\end{pmatrix}\rrangle[\Big] dt
				-\int_{n}\left\langle \bm{A}_h \bm{T}_{I}^{n},\Pi_{\t}^{k-1}\bmE_{\t,h}^{2}\right\rangle dt
				\\&\quad-\int_{I_n}\langle\bm{A}_h\p_t\bm{T}_{II}^{n},\bmE_{\t,h}^{1}\rangle dt
				-	\int_{I_n}\llangle[\Big] \bm{N}\begin{pmatrix}
					\p_t\bm{\eta}_{2}\\
					\p_t\bm{\zeta}
				\end{pmatrix},\begin{pmatrix}
					\Pi_{\t}^{k-1}\bmE_{\t,h}^{2}\\
					\Pi_{\t}^{k-1}\bmEth^{3}
				\end{pmatrix}\rrangle[\Big]dt-\a\int_{I_n}\langle \p_t\om,\div\bmE_{\t,h}^{1}\rangle dt
				\\
				&\quad-c_0\int_{I_n}\langle \p_t\om,\Pi_{\t}^{k-1}e_{\t,h}\rangle dt
				-\a\int_{I_n}\langle \div\p_t\bm{\eta}_1,\Pi_{\t}^{k-1}e_{\t,h}\rangle dt
				-\int_{I_n}\langle \div\bm{\zeta},\Pi_{\t}^{k-1}e_{\t,h}\rangle dt
				\\
				&\quad
				-\rho_f Fd_1^{-1}\eta_k^{-1}\int_{I_n}\langle \bm{\zeta},\Pi_{\t}^{k-1}\bmEth^{3}\rangle dt
				+\int_{I_n}\langle \om,\div\Pi_{\t}^{k-1}\bmEth^{3}\rangle dt
				\\&\quad +\delta_{1,n}-\delta_{1,n-1}^{+}+\delta_{2,n}-\delta_{2,n-1}^{+}.
			\end{aligned}
		\end{equation}
		First, we observe that \eqref{eq:errorVelocity} with $\bmPth^{1}=\p_{t}\bmE_{\t,h}^{1}$ implies that $\|\bmEth^{2}\|_{L^{2}(I_n;\bm{L}^{2})}=\|\p_t\bmEth^{1}\|_{L^{2}(I_n;\bm{L}^{2})}$. Further, by the $H^{1}-L^{2}$ inverse inequality \eqref{eq:HoneLtwoInverseIneq} we see that $$\norm{\bmEth^{2}}_{L^{2}(I_n;\bm{L}^{2})}=\norm{\p_t\bmEth^{1}}_{L^{2}(I_n;\bm{L}^{2})}\leq c\t_n^{-1}\norm{\bmEth^{1}}_{L^{2}(I_n;\bm{L}^{2})}.$$  Then, by the estimates in \Cref{lemma:estimation} along with Cauchy-Schwarz and Young inequalities with a sufficiently small $\e>0$ we obtain the stability estimate
		\eqref{eq:stabilityestimate}.

	\end{proof}
	\begin{lemma}\label{lemma:newname}
		Let $n=1,...,N$. For the errors $\bmE_{\t,h}^{1}, \bmE_{\t,h}^{2}, e_{\t,h}$ and $\bmEth^{3}$ defined in \eqref{eqs:errorsplit}, there holds that
		\begin{equation}\label{eq:estimate}
			\begin{aligned}
				c&\left\| \bm{N}^{1/2}\begin{pmatrix}
					\bmE_{\t,h}^{2}\\
					\bmEth^{3}
				\end{pmatrix}\right\|^{2}_{L^{2}(I_n;\bm{L}^{2})}+c\norm{\bmE_{\t,h}^{1}}_{L^{2}(I_n;\bm{L}^{2})}^{2} +c\|e_{\t,h}\|^{2}_{L^{2}(I_n;L^{2})}
				\\\quad&\leq c\t_n \left\| \bm{N}^{1/2}\begin{pmatrix}
					\bmE_{\t,h}^{2}(t_{n-1}^{+})\\
					\bmEth^{3}(t_{n-1}^{+})
				\end{pmatrix}\right\|_{\bm{L}^{2}}^{2}+c\t_n\norm{\bmE_{\t,h}^{1}(t^{+}_{n-1})}^{2}+c\t_n\|e_{\t,h}(t^{+}_{n-1})\|_{L^{2}}^{2}
				\\\quad\quad&+  c\t_n\left(\t_{n}^{2(k+1)}(\varepsilon_{t}^{n,1})^{2}+h^{2(r)}(\varepsilon_{x}^{n,0})^{2}\right),
			\end{aligned}
		\end{equation}
		where $\varepsilon_{t}^{n,1}$ and $\varepsilon_{x}^{n,0}$ are defined in \eqref{eq:notationsummation} and $\varepsilon_{x}^{n,2}=\|p\|_{L^{\infty}(I_n;H^{r+1})}$.
	\end{lemma}
	\begin{proof}
		Similarly to the previous lemma, we need to choose the test functions so that the terms $\mathfrak{B}_{\eta_{k}}^{P}$ and $\mathfrak{B}_{\eta_{k}}^{\Psi}$ and the terms $\mathfrak{A}_{\a}^{P}$ and $\mathfrak{A}_{\a}^{V}$ in the summed error equation \eqref{eq:errorequationsadded} cancel. In addition, we use an idea from \cite{makridakis1998} that we consider representations of the errors on the subinterval $I_n$. We let the mechanical error  $\bmE_{\t,h}=\left(\bmE_{\t,h}^{1},\bmE_{\t,h}^{2}\right)$ and the ADE error $\bmEth^{3}$ be represented by 
		\begin{align}\label{eq:representation}
			\bmE_{\t,h}^{m}(t)=\sum_{j=0}^{k}\bmE_{n,j}^{m}\phi_{n,j}(t),\quad\mbox{for } t\in I_n,\,m\in\{1,2,3\},
		\end{align}
		where $\bmE_{n,j}^{m}\in\bm{V}^{r}_{h}$ for $j=0,...,k$, and $\phi_{n,j}\in\P_{k}(I_n;\R)$ for $j=0,...k$ are the Lagrange interpolants with respect to $t_{n-1}$ and the Gauss quadrature nodes $t_{n,1}^{G},t_{n,2}^{G},...,t_{n,k}^{G}\in (t_{n-1},t_n)$ of \eqref{gaussquadratureInKthpoint}. Consequently, by recalling the limit definition $\eqref{eq:limitdef}$, it holds that $\bmE_{n,0}^{m}=\bmE_{\t,h}^{m}(t_{n-1}^{+})$. Also for the pressure error $e_{\t,h}$ we use the representation
		\begin{align}\label{eq:representation2}
			e_{\t,h}(t)=\sum_{j=0}^{k}e_{n,j}\phi_{n,j}(t),\quad \mbox{for }t\in I_n,
		\end{align}
		where $e_{n,j}\in V^{r}_{h},$ for $j=0,...,k$.  Let the values $\hat{t}_{i}^{G}$ denote the quadrature nodes of the Gauss formula \eqref{gaussquadratureInKthpoint} on a reference element $\hat{I}$ and by considering \eqref{eq:errorequationsadded} we choose the test functions 
		\begin{equation}\label{eq:testfunctiontilde}
			\begin{aligned}
				\bmPth^{2}(t)=\sum_{i=1}^{k}(\hat{t}_{i}^{G})^{-1/2}\widetilde{\bmE}_{n,i}^{2}\psi_{n,i}(t),&\quad \bmPth^{3}(t)=\sum_{i=1}^{k}(\hat{t}_{i}^{G})^{-1/2}d^{-1}_1\eta_k^{-1}\widetilde{\bmE}_{n,i}^{3}\psi_{n,i}(t),\\ \psi_{\t,h}(t)=&\sum_{i=1}^{k}(\hat{t}^{G}_{i})^{-1/2}\widetilde{e}_{n,i}\psi_{n,i}(t),
			\end{aligned}
		\end{equation}
		where $\widetilde{\bmE}_{n,i}^{m}=(\hat{t}_{i}^{G})^{-1/2}\bmE_{n,i}^{m}$ for $m\in\{1,2,\Psi\}$ and $\widetilde{e}_{n,i}:=(\hat{t}^{G}_{i})^{-1/2}e_{n,i}$ for $i=1,...,k$, and $\psi_{n,i}\in\P_{k-1}(I_n;\R)$ for $i=1,...,k$ are the Lagrange interpolants with respect to the Gauss quadrature nodes $t_{n,1}^{G},t_{n,2}^{G},...,t_{n,k}^{G}\in (t_{n-1},t_n)$ of \eqref{gaussquadratureInKthpoint}. First, we consider the red terms $\mathfrak{A}_{\a}^{P}$ and $\mathfrak{A}_{\a}^{V}$ in \eqref{eq:errorequationsadded}. By the observation on a similar coupling term in \cite[Eq. (4.51)]{bause2024} we see that for our choice of test functions, it will cancel.
	In addition, the blue terms $\mathfrak{B}_{\eta_{k}}^{P}$ and $\mathfrak{B}_{\eta_{k}}^{\Psi}$ cancel by the choice of test functions \eqref{eq:testfunctiontilde}.
	Next, we consider the first term of \eqref{eq:errorequationsadded}.
	To control this term, we use an idea similar to \cite[Eq. (3.25)]{makridakis2005} and \cite[Lemma 3.2]{makridakis1998}. First, we define the matrix $\bm{M}=(m_{ij})_{i,j=1,...,k}$ and the vector $\bm{m}_0=(m_{i0})_{i=1,...,k}$ where
	\begin{align}\label{eq:matrixMij}
		m_{ij}:=\int_{I_n}\phi_{n,j}'(t)\psi_{n,i}(t)dt,\mbox{ for } i =1,...,k, \mbox{ for } j=1,...,k, \quad m_{i0}:= \int_{I_n}\phi_{n,0}'(t)\psi_{n,i}(t)dt,
	\end{align}
	for $i =1,...,k$ and further the matrix $\widetilde{\bm{M}}= (\widetilde{m}_{ij})_{i,j=1,...,k}$ is given by
	\begin{align*}
		\widetilde{\bm{M}}:=\bm{D}^{-1/2}\bm{M}\bm{D}^{1/2}, \mbox{ with } \bm{D}={\rm diag}\{\hat{t}^{G}_{1},...,\hat{t}^{G}_{k}\}.
	\end{align*}
	Then, we use the observation that $\bmE_{n,0}^{m}=\bmE_{\t,h}^{m}(t_{n-1}^{+})$, for $m\in\{1,2,\Psi\}$ along with the test function \eqref{eq:testfunctiontilde}. We get that 
	\begin{equation}\label{eq:chiNdefinition}
		\begin{aligned}
			\chiup_n:=	\int_{I_n}\llangle[\Big] \bm{N}\begin{pmatrix}
				\p_t\bmE^{2}_{\t,h}\\
				\p_t\bmE^{3}_{\t,h}
			\end{pmatrix},\begin{pmatrix}
				\bmPth^{2}\\
				\bmPth^{3}
			\end{pmatrix}\rrangle[\Big]dt=\int_{I_n}\llangle[\Big] \bm{N}\begin{pmatrix}
				\p_t\bmE^{2}_{\t,h}\\
				\p_t\bmE^{3}_{\t,h}
			\end{pmatrix},\sum_{i=1}^{k}(\hat{t}_{i}^{G})^{-1/2}\begin{pmatrix}
				\widetilde{\bmE}_{n,i}^{2}\\
				\widetilde{\bmE}_{n,i}^{3}
			\end{pmatrix}\psi_{n,i}\rrangle[\Big]dt
			\\
			=\int_{I_n}\llangle[\Big] \bm{N}\sum_{j=0}^{k}(\hat{t}_{i}^{G})^{1/2}\begin{pmatrix}
				\widetilde{\bmE}_{n,j}^{2}\\
				\widetilde{\bmE}_{n,j}^{3}
			\end{pmatrix}\phi_{n,j}'(t),\sum_{i=1}^{k}(\hat{t}_{i}^{G})^{-1/2}\begin{pmatrix}
				\widetilde{\bmE}_{n,i}^{2}\\
				\widetilde{\bmE}_{n,i}^{3}
			\end{pmatrix}\psi_{n,i}\rrangle[\Big]dt
			\\
			=\sum_{i,j=1}^{k}\widetilde{m}_{ij}\llangle[\Big] \bm{N}\begin{pmatrix}
				\widetilde{\bmE}_{n,j}^{2}\\
				\widetilde{\bmE}_{n,j}^{3}
			\end{pmatrix},\begin{pmatrix}
				\widetilde{\bmE}_{n,i}^{2}\\
				\widetilde{\bmE}_{n,i}^{3}
			\end{pmatrix}\rrangle[\Big]+\sum_{i=1}^{k}m_{i0}(t_i^{G})^{-1/2}\llangle[\Big] \bm{N}\begin{pmatrix}
				\bmE_{\t,h}^{2}(t_{n-1}^{+})\\
				\bmEth^{3}(t_{n-1}^{+})
			\end{pmatrix},\begin{pmatrix}
				\widetilde{\bmE}_{\t,h}^{2}\\
				\widetilde{\bmE}_{\t,h}^{3}
			\end{pmatrix}\rrangle[\Big].
		\end{aligned}
	\end{equation}
	Then, noting that the positive definiteness of $\widetilde{\bm{M}}$ is ensured by \cite[Lem. 2.1]{makridakis2005}, we get that
	\begin{equation*}
		\begin{aligned}
			\chiup_n&\geq c\sum_{j=1}^{k}\left\| \bm{N}^{1/2}\begin{pmatrix}
				\widetilde{\bmE}_{n,j}^{2}\\
				\widetilde{\bmE}_{n,j}^{3}
			\end{pmatrix}\right\|^{2}-c\left(\sum_{j=1}^{k}\left\| \begin{pmatrix}
				\widetilde{\bmE}_{n,j}^{2}\\
				\widetilde{\bmE}_{n,j}^{3}
			\end{pmatrix}\right\|^{2}\right)^{1/2}\left\| \bm{N}^{1/2}\begin{pmatrix}
				\bmE_{\t,h}^{2}(t_{n-1}^{+})\\
				\bmEth^{3}(t_{n-1}^{+})
			\end{pmatrix}\right\|
			\\
			&\geq c\sum_{j=1}^{k}\left\| \bm{N}^{1/2}\begin{pmatrix}
				\widetilde{\bmE}_{n,j}^{2}\\
				\widetilde{\bmE}_{n,j}^{3}
			\end{pmatrix}\right\|^{2}-c\left\| \bm{N}^{1/2}\begin{pmatrix}
				\bmE_{\t,h}^{2}(t_{n-1}^{+})\\
				\bmEth^{3}(t_{n-1}^{+})
			\end{pmatrix}\right\|^{2}.
		\end{aligned}
	\end{equation*}
	Furthermore, by the equivalence of $\sum_{j=1}^{k}\|\widetilde{\bmE}_{n,j}\|^{2}$ and $\sum_{j=1}^{k}\|\bmE_{n,j}\|^{2}$ and the norm equivalence \cite[Eq.\ 2.4]{makridakis2005} we see that
	\begin{equation}\label{eq:chiNestimate}
		\t_n\chiup_n\geq c \left\| \bm{N}^{1/2}\begin{pmatrix}
			\bmE_{\t,h}^{2}\\
			\bmEth^{3}
		\end{pmatrix}\right\|^{2}_{L^{2}(I_n;\bm{L}^{2})}-c\t_n \left\| \bm{N}^{1/2}\begin{pmatrix}
			\bmE_{\t,h}^{2}(t_{n-1}^{+})\\
			\bmEth^{3}(t_{n-1}^{+})
		\end{pmatrix}\right\|_{\bm{L}^{2}}^{2}.
	\end{equation}
	We must also control the second term on the left-hand side of \eqref{eq:errorequationsadded}. First, note the similarity with the terms in \cite[Eq. (4.46)]{bause2024} and \cite[Eq. (69)]{kraus2024} along with the observation \eqref{eq:error2topartialerror1} from which it follows that
	\begin{equation}\label{eq:Qn}
		\begin{aligned}
			Q_n:=&\int_{I_n}\langle \bm{A}_{h}\bmE^{1}_{\t,h},\bmPth^{2}\rangle dt = \int_{I_n}\langle \bm{A}_{h}\bmE^{1}_{\t,h},\sum_{i=1}^{k}(\hat{t}_{i}^{G})^{-1/2}\widetilde{\bmE}_{n,i}^{2}\psi_{n,i}(t)\rangle dt
			\\
			=&\sum_{i,j=1}^{k}\widetilde{m}_{ij}\langle \bm{A}_{h}\widetilde{\bmE}^{1}_{n,j},\widetilde{\bmE}_{n,i}^{1}\rangle+\sum_{i=1}^{k}m_{i0}(t_i^{G})^{-1/2}\langle \bm{A}_{h}\bmE^{1}_{\t,h}(t_{n-1}^{+}),\widetilde{\bmE}_{n,i}^{1}\rangle.
		\end{aligned}
	\end{equation}
	Again by the positivity of $\widetilde{\bm{M}}$ we have that
	\begin{equation}\label{eq:QnEstimate}
		\t_n Q_n\geq \norm{\bmE_{\t,h}^{1}}_{L^{2}(I_n;\bm{L}^{2})}^{2}-c\t_n\norm{\bmE_{\t,h}^{1}}^{2}.
	\end{equation}
From \cite[Lem 4.5]{bause2024} with test function \eqref{eq:testfunction} we already have the estimate
\begin{equation}\label{eq:RnEstimate}
	\begin{aligned}
		\t_n\int_{I_n}\langle \om,\div\bmPth^{2}\rangle dt \leq c\t_nh^{2(r+1)}\|p\|_{L^{\infty}(I_n;H^{r+1})}^{2}+\e\norm{\bmE_{\t,h}^{1}}^{2}_{L^{2}(I_n;\bm{L}^{2})},
	\end{aligned}
\end{equation}
where $\e>0$ is a sufficiently small constant resulting from Young's inequality.

Next, we consider the error equation for the flow \eqref{eq:errorFlow}. Note that the same terms appear in \cite[Eq. (4.55)]{bause2024} and therefore we have
\begin{equation}\label{eq:SnEstimate}
	\begin{aligned}
		&\,\t_n\int_{I_n}\langle \p_t e_{\t,h},\psi_{\t,h}\rangle dt \geq c\|e_{\t,h}\|^{2}_{L^{2}(I_n;L^{2})}-c\t_n\|e_{\t,h}(t_{n-1}^{+})\|_{L^{2}}^{2}.
	\end{aligned}
\end{equation}
	Finally, considering the summation of the errors  \eqref{eq:errorequationsadded} multiplied with $\t_n$ and by choosing the test functions \eqref{eq:testfunctiontilde} to cancel out the blue and red terms we can use the inequalities \eqref{eq:chiNestimate}, \eqref{eq:QnEstimate}, \eqref{eq:RnEstimate}, and \eqref{eq:SnEstimate} along with \Cref{lemma:estimation} and Young's inequality to obtain the estimate \eqref{eq:estimate}.

		\end{proof}
		
		\begin{lemma}\label{lemma:estimation2}
			Let $n\in\{2,...,N\}$. For the errors $\bmE_{\t,h}^{1},\bmE_{\t,h}^{2}, e_{\t,h}$ and $\bmEth^{3}$ defined in \eqref{eqs:errorsplit}, the following inequality holds
			\begin{equation}\label{eq:lemma49}
				\begin{aligned}
					&\norm{\bmE_{\t,h}^{1}(t_{n-1}^{+})}_e+\frac{c_0}{2}\|e_{\t,h}(t_{n-1}^{+})\|_{L^{2}}^{2}+\frac12\left\|\bm{N}^{\frac12}\begin{pmatrix}
						\bmE^{2}_{\t,h}(t_{n-1}^{+})\\\bmEth^{3}(t_{n-1}^{+})
					\end{pmatrix}\right\|_{\bm{L}^{2}}^{2}
					\\
					\leq& (1+\t_{n-1})\left(\norm{\bmE_{\t,h}^{1}(t_{n-1})}_e+\frac{c_0}{2}\|e_{\t,h}(t_{n-1})\|_{L^{2}}^{2}+\frac12\left\|\bm{N}^{\frac12}\begin{pmatrix}
						\bmE^{2}_{\t,h}(t_{n-1})\\\bmEth^{3}(t_{n-1})
					\end{pmatrix}\right\|_{\bm{L}^{2}}^{2}\right)
					\\
					&\,+c\t_{n-1}^{2(k+1)}(\varepsilon_{t}^{n-1,2})^{2},
				\end{aligned}
			\end{equation}
			where $\varepsilon_t^{n-1,2}=\|\p_t^{k+2}\bmu\|_{L^{2}(I_{n-1};\bm{H}^{1})}$. It also holds that the errors $\bmE_{\t,h}^{1},\bmE_{\t,h}^{2}, e_{\t,h}$ and $\bmEth^{3}$ at $t^{+}_{0}$ satisfy
			\begin{equation}\label{eq:initailrel}\begin{aligned}
					\norm{\bmE_{\t,h}^{1}(t_0^{+})}_e^{2}&+\frac12\left\|\bm{N}^{\frac12}\begin{pmatrix}
						\bmE^{2}_{\t,h}(t_0^{+})\\\bmEth^{3}(t_0^{+})
					\end{pmatrix}\right\|_{\bm{L}^{2}}^{2}+\frac{c_0}{2}\|e_{\t,h}(t_0^{+})\|_{L^{2}}^{2}	
					\leq ch^{2r}.
				\end{aligned}
			\end{equation}
		\end{lemma}
		\begin{proof}
			The proof is in \ref{appA}.
		\end{proof}
		
		\begin{lemma}\label{lemma:delta1}
			Let $\d_{1,n}$ and $\d_{1,n-1}^{+}$ be defined by \eqref{eqs:deltasinTheorem}. For $n=2,...,N$ it holds that
			\begin{align}
				\d_{1,n}-\d_{1,n-1}^{+}\leq \d_{1,n}-\d_{1,n-1}+c\t_{n-1}\t_{n-1}^{2(k+1)}((\varepsilon_{t}^{n-1,4})^{2}+(\varepsilon_{t}^{n-1,3})^{2}),
			\end{align}
			where $\varepsilon_{t}^{n-1,3}=\|\p_t^{k+2}\bmu\|_{L^{\infty}(I_{n-1};\bm{H}^{1})}$ and $\varepsilon_{t}^{n-1,4}=\|\p_t^{k+1}\bmu\|_{L^{\infty}(I_{n-1};\bm{H}^{1})}$. When $n=1$, it holds that 
			\begin{equation}\label{eq:deltainitial1}
				\begin{aligned}
					|\d_{1,1}-\d_{1,0}^{+}|\leq&\, c\t^{2(k+1)}\|\p_t^{k+1}\bmu_0\|^{2}_{1}+ch^{2(r+1)}\|\bmu_0\|_{r+1}^{2}\\
					&\,+c\t^{2(k+1)}\|\p_t^{k+1}\bmu(t_1)\|^{2}_{1}+\e\|\del\bmEth^{1}(t_1)\|^{2},
				\end{aligned}
			\end{equation}
			for a sufficiently small $\e>0$. 
		\end{lemma}
		\begin{proof}
			The proof is in Appendix \ref{appA}.
		\end{proof}
		
		Estimates for $\d_{2,n}$ and $\d_{2,n-1}^{+}$ for non-equal order inf-sup stable elements can be found in \cite[Lemma 4.7]{bause2024}. The corresponding estimate for equal-order elements is only order $h^{r}$ in space and can be shown similarly.
		\begin{lemma}\label{lemma:delta}
			Let $\d_{2,n}$ and $\d_{2,n-1}^{+}$ be defined by \eqref{eqs:deltasinTheorem}. For $n=2,...,N$ it holds that
			\begin{align}
				\d_{2,n}-\d_{2,n-1}^{+}\leq \d_{2,n}-\d_{2,n-1}+c\t_{n-1}\t_{n-1}^{2(k+1)}(\varepsilon_{t}^{n-1,3})^{2}+c\t_{n-1}h^{2(r)}(\varepsilon_{x}^{n-1,2})^{2},
			\end{align}
			where $\varepsilon_{t}^{n-1,3}$ is defined in \Cref{lemma:delta1} and $\varepsilon_{x}^{n-1,2}$ is defined in \Cref{lemma:newname}. When $n=1$, it holds that 
			\begin{equation}\label{eq:deltainitial2}
				|\d_{2,1}-\d_{2,0}^{+}|\leq ch^{2(r)}(h\|p_0\|^{2}_{r+1}+h\|p(t_1)\|_{r+1}^{2}+\|\bmu_0\|^{2}_{r+1})+\e\norm{\bmE_{\t,h}^{1}(t_1)}_e^{2},
			\end{equation}
			for a sufficiently small $\e>0$. 
		\end{lemma}

		\subsection{Main error estimate}
		We begin by stating the Gronwall inequality used in the main error estimate.
		\begin{lemma}[Discrete Gronwall inequality {\cite[Lemma 1.4.2]{quarteroni1994}}]
			Assume $k_n$ is a non-negative sequence, and that the sequence $\phi_n$ satisfies
			\begin{align}
				\phi_0 &\leq g_0, \\
				\phi_n&\leq g_0+\sum_{s=0}^{n-1}p_s+\sum_{s=0}^{n-1}k_s\phi_s,\quad n\geq 1. 
			\end{align}
			Then $\phi_n$ satisfies
			\begin{align}
				\phi_1 &\leq g_0(1+k_0)+p_0, \\
				\phi_n&\leq g_0\prod_{s=0}^{n-1}(1+k_s)+\sum_{s=0}^{n-2}p_s\prod_{j=s+1}^{n-1}(1+k_j)+p_{n-1},\quad n\geq 2. 
			\end{align}
			Moreover, if $g_0\geq 0$ and $p_n\geq 0$ for $n\geq 0$, it follows that
			\begin{align}
				\phi_{n}\leq (g_0+\sum_{s=0}^{n-1}p_s \exp{(\sum_{s=0}^{n-1}k_s)},\quad n\geq 1.
			\end{align}
		\end{lemma}
		
		Now, we can prove the main result of this work.
		\begin{theorem}\label{maintheorem}
			For the approximation $(\bmu_{\t,h},\bmv_{\t,h},p_{\t,h},\Psi_{\t,h})$ defined by \Cref{prob:quadratureform} of a sufficiently smooth solution to \eqref{eqs:biotallardADE}, the following estimate holds
			\begin{align}\label{eq:mainestimate}
				\|\del(\bmu(t)-\bmu_{\t,h}(t))\|+\left\|\bm{N}^{\frac12}\begin{pmatrix}
					\bmv(t)-\bmv_{\t,h}(t)\\\Psi(t)-\Psi_{\t,h}(t)
				\end{pmatrix}\right\|+c_0\|p(t)-p_{\t,h}(t)\|\leq c\left(\t^{k+1}+h^{r}\right),
			\end{align}
			for $t\in I$.
		\end{theorem}
		\begin{proof}
			First, we denote 
			\begin{align*}
				A_n:=\norm{\bmE_{\t,h}^{1}(t_{n})}_e^{2}+\frac12\left\|\bm{N}^{\frac12}\begin{pmatrix}
					\bmE^{2}_{\t,h}(t_n)\\\bmEth^{3}(t_n)
				\end{pmatrix}\right\|_{\bm{L}^{2}}^{2}+\frac{c_0}{2}\|e_{\t,h}(t_n)\|_{L^{2}}^{2},\quad \mbox{for } n=0,1,...,N.
			\end{align*}
			By combing the stability estimate \Cref{theorem:stability} and \Cref{lemma:newname} and the norm equivalence \eqref{eq:norm} we get
			\begin{equation}\label{eq:maintheoremstepOne}
				\begin{aligned}
					A_n	&+C_{\e}\|\bmEth^{3}\|_{L^{2}(I_n:\bm{L}^{2})}^{2}\\
					&\leq (1+c\t_n)\left(\norm{\bmE_{\t,h}^{1}(t^{+}_{n-1})}_e^{2}+\frac{c_0}{2}\|e_{\t,h}(t_{n-1}^{+})\|_{L^{2}}^{2}+\frac{1}{2}\left\|\bm{N}^{\frac12}\begin{pmatrix}
						\bmE^{2}_{\t,h}(t_{n-1}^{+})\\\bmEth^{3}(t_{n-1}^{+})
					\end{pmatrix}\right\|_{\bm{L}^{2}}^{2}\right)\\
					&\quad+\delta_{1,n}-\delta_{1,n-1}^{+}+\delta_{2,n}-\delta_{2,n-1}^{+}
					+c(1+\t_n)\left(\t_{n}^{2(k+1)}(\varepsilon_{t}^{n,1})^{2}+ch^{2(r)}(\varepsilon_{x}^{n,0})^{2}\right)
				\end{aligned}
			\end{equation}
			for $n=1,...,N$.
			Then by \eqref{eq:lemma49}, \Cref{lemma:delta1} and \Cref{lemma:delta}, we obtain for $n=2,...,N$
			\begin{equation}\label{eq:maintheoremstepTwo}
				\begin{aligned}
					A_n
					&\leq(1+c\t_n)(1+\t_{n-1})A_{n-1}
					+\delta_{1,n}-\delta_{1,n-1}+\delta_{2,n}-\delta_{2,n-1}
					\\&\quad+c(1+\t_n)\left(\t_{n}^{2(k+1)}(\varepsilon_{t}^{n,1})^{2}+ch^{2(r)}(\varepsilon_{x}^{n,0})^{2}\right)
					+c\t_{n-1}\t_{n-1}^{2(k+1)}((\varepsilon_{t}^{n-1,4})^{2}+(\varepsilon_{t}^{n-1,3})^{2})\\
					&\quad+c\t_{n-1}\t_{n-1}^{2(k+1)}(\varepsilon_{t}^{n-1,3})^{2}+c\t_{n-1}h^{2(k+1)}(\varepsilon_{x}^{n-1,2})^{2}.
				\end{aligned}
			\end{equation}
			Next, we consider the case of $n=1$. 
			Then by \eqref{eq:deltainitial1}, \eqref{eq:deltainitial2} and \eqref{eq:initailrel} inserted into \eqref{eq:maintheoremstepOne} for sufficiently smooth solutions we have
			\begin{equation}\label{eq:maintheoremstepFour}
				A_1=\,\norm{\bmE_{\t,h}^{1}(t_1)}_e^{2}+\frac12\left\|\bm{N}^{\frac12}\begin{pmatrix}
					\bmE^{2}_{\t,h}(t_1)\\\bmEth^{3}(t_1)
				\end{pmatrix}\right\|_{\bm{L}^{2}}^{2}+\frac{c_0}{2}\|e_{\t,h}(t_1)\|_{L^{2}}^{2}	
				\leq c\left(\t_1^{2(k+1)}+h^{2r}\right).
			\end{equation}
			Next, we apply the discrete Gronwall inequality \cite[Lemma 1.4.2]{quarteroni1994} to \eqref{eq:maintheoremstepTwo} and \eqref{eq:maintheoremstepFour}. Therefore, we sum up the inequality \eqref{eq:maintheoremstepTwo} from $j=2$ to $j=n$, this gives
			\begin{equation}
				\begin{aligned}
					A_{n}&\leq |\d_{2,1}|+|\d_{2,n}|+|\d_{1,1}|+|\d_{1,n}|+\sum_{j=2}^{n}(c\t_{j}+\t_{j-1}+c\t_{j}\t_{j-1})A_{j-1}\\
					&\quad+(1+\t)\left(c\t^{2(k+1)}+ch^{2(r)}\right)\left(\sum_{j=1}^{n}\left((\varepsilon_x^{n,0})^{2}+(\varepsilon_t^{n,1})^{2}\right)\right).
				\end{aligned}
			\end{equation}
			Further, by \eqref{eq:notationsummation}, the last term on the right-hand side is smaller than or equal to a constant $c< \infty$.
			Now combining \eqref{eq:delta11} with \eqref{eq:maintheoremstepFour} we see that
			\begin{subequations}\label{eq:deltainit}
				\begin{equation}
					|\delta_{1,1}|\leq c\left(\t_1^{2(k+1)}+h^{2(r)}\right).
				\end{equation}
				Similarly to \cite[Eq. (4.80)]{bause2024}, we also have 
				\begin{equation}
					|\delta_{2,1}|\leq c\left(\t_1^{2(k+1)}+h^{2(r)}\right).
				\end{equation}
			\end{subequations}
			Furthermore, by \eqref{eqs:deltasinTheorem}, \eqref{eqs:errorsplit}, and the Cauchy-Schwarz and Young inequalities with $\e>0$ sufficiently small along with \eqref{eq:ellipticprojectionerrorScalar} and \eqref{eq:lagrangestability} it holds that
			\begin{equation}
				\begin{aligned}
					|\d_{1,n}|\leq \,c\t_n^{2(k+1)}+\e A_n,\quad \mbox{ and }\quad
					|\d_{2,n}|\leq \,ch^{2(r+1)}+\e A_n.
				\end{aligned}
			\end{equation}
			We also have
			\begin{equation}\label{eq:expgronwall}
				\begin{aligned}
					\prod_{j=1}^{n-1}(1+c\t_j)\leq& \,e^{cT}.
				\end{aligned}
			\end{equation}
			Then, combining the Gronwall argument with \eqref{eq:deltainit} to \eqref{eq:expgronwall} and assumption \eqref{ass:initialapprox} results in
			\begin{equation}\label{eq:gronwallconclusion}
				\begin{aligned}
					&\norm{\bmE_{\t,h}^{1}(t_{n})}_e^{2}+\frac12\left\|\bm{N}^{\frac12}\begin{pmatrix}
						\bmE^{2}_{\t,h}(t_n)\\\bmEth^{3}(t_n)
					\end{pmatrix}\right\|_{\bm{L}^{2}}^{2}+\frac{c_0}{2}\|e_{\t,h}(t_n)\|_{L^{2}}^{2}	\leq c\left(\t^{2(k+1)}+h^{2(r)}\right),
				\end{aligned}
			\end{equation}
			for $n=0,...,N$, where $\t=\max_{n\in\{1,...,N\}}\t_n$. By \Cref{lemma:estimation2}, \Cref{lemma:newname}, \,\eqref{eq:gronwallconclusion} and the norm equivalence \eqref{eq:norm}, we obtain for $n=2,...,N$
			\begin{equation}
				\begin{aligned}
					&\norm{\bmE_{\t,h}^{1}}_{L^{2}(I_n;\bm{L}^{2})}^{2}+\left\|\bm{N}^{\frac12}\begin{pmatrix}
						\bmE^{2}_{\t,h}\\\bmEth^{3}
					\end{pmatrix}\right\|^{2}_{L^{2}(I_n;\bm{L}^{2})}+c_0\|e_{\t,h}\|^{2}_{L^{2}(I_n;L^{2})}	\leq c\t\left(\t^{2(k+1)}+h^{2(r)}\right).
				\end{aligned}
			\end{equation}
			When $n=1$, we obtain the above estimate by using \Cref{lemma:newname} combined with \eqref{eq:initailrel} and \eqref{eq:norm}. Then the $L^{\infty}-L^{2}$ relation \eqref{eq:LInfinityLTwo} gives
			\begin{equation}
				\norm{\bmE_{\t,h}^{1}(t)}^{2}+\left\|\bm{N}^{\frac12}\begin{pmatrix}
					\bmE^{2}_{\t,h}(t)\\\bmEth^{3}(t)
				\end{pmatrix}\right\|^{2}+c_0\|e_{\t,h}(t)\|^{2}\leq c\left(\t^{2(k+1)}+h^{2(r)}\right),\quad \mbox{for }t\in[0,T].
			\end{equation}
			Then, by applying the triangle inequality to the splitting of the errors \eqref{eqs:errorsplit} and using the estimates in \Cref{lemma:splittingestimate}, we obtain the estimate \eqref{eq:mainestimate}.
			
		\end{proof}
		
		\begin{remark}
			We note that the error estimate \eqref{eq:mainestimate} is of optimal order with respect to the approximation in space and time. Here, the error is measured in terms of the first-order energy of the variables $\{\boldsymbol u, \boldsymbol v, p, \Psi\}$ of the system \eqref{eqs:firstorderADE}. In particular, an error bound for the displacement field $\boldsymbol u$ in the $\boldsymbol{H}^1$ seminorm is part of the estimate. The error norm occurs naturally in our energy-type error analysis and is due to a stability argument similarly to the one of \eqref{eq:stabilityestimate}. The error constant $c$ depends on norms of the continuous solutions that are introduced in \Cref{lemma:splittingestimate} and \Cref{lemma:estimation} to \Cref{lemma:delta}. These norms, as well as the weighted energy norm used in the error estimate, may depend on the model parameters. However, the constant $c$ does not depend on the discretization parameters $h$ and $\tau$.
		\end{remark}
		
		\begin{remark}[Non-equal order approximation]\label{rem:infsup}\,
			\begin{itemize}
				\item The discrete schemes of Problem~\ref{prob:quadratureform} and \ref{Prob2} are based on an equal-order approximation of the variables. For vanishing coefficients $c_0\rightarrow 0$ and $|\calA| \rightarrow  0$, a Stokes-type system structure is obtained in \eqref{eqs:firstorderADE} for the variables $\partial_t \bmu$ and $p$ such that the well-known stability issues of mixed approximations of the Stokes system emerge; cf.~\cite{volker2016}. In such a case, equal-order spatial discretizations do not become feasible without additional stabilization of the discretization. Also, structure-preserving first-order in space and time approaches might be applied; cf. \cite{hong2018,kanschat_finite_2018}.  For a more detailed discussion of stability properties for the quasi-static Biot system, we also refer to, e.g., \cite{hong2018,murad1992,murad1994,murad_asymptotic_1996,rodrigo2018}.
				\item Using a non-equal order approximation in space for $(\bmu,\bmv, p,\Psi)\in (\bm{V}_h^{r+1}\times\bm{V}_h^{r+1}\times V_h^{r}\times\bm{V}_h^{r+1})$ built upon the Taylor-Hood pair of elements satisfies the inf-sup stability condition. In \Cref{sec:numerics}, we have included numerical experiments with inf-sup stable elements for the sake of completeness. 
				\item Based on the numerical results in \Cref{table:error} and \Cref{table:testcase2InfSup}, we conjecture that the error estimate \eqref{eq:conjecture} holds for non-equal order inf-sup stable elements. By comparing the estimates in \Cref{lemma:estimation} with the conjecture \eqref{eq:conjecture}, it is apparent that the gradient of the interpolation error of the pressure, i.e. $\langle\del\om,\bmPth^{3}\rangle$, causes a loss of one spatial order in the analysis. We expect that proving the conjecture can be done by building upon \Cref{maintheorem}. The proof of the conjecture remains an open problem and is left as work for the future. 
			\end{itemize}
		\end{remark}

		\section{Numerical results}\label{sec:numerics}
		In this section, we perform numerical tests to verify our main theoretical result \Cref{maintheorem}. The implementation is done in FreeFEM++ \cite{freefem} as a Crank-Nicolson scheme in time, which is equivalent to the lowest-order time discretization considered here. Every linear system is solved using a direct solver. We consider two test cases using the same manufactured solution, both in an isotropic and homogeneous medium. In the second test case, pressure is scaled by $p_{\!_{\rm ref}}$ to balance the magnitude of the mechanical and fluid stresses for the chosen physical parameters. The problem \eqref{eqs:biotallardADE} is considered on $\Om=(0,1)^{2}$ and $I=(0,0.1]$ for the prescribed solution
		\begin{equation}\label{eq: manufactured sol 1}
			\Psi= \bmu = \phi(x,t)\textbf{I},\quad  \frac{1}{p_{\!_{\rm ref}}}p=\phi(x,t), 
		\end{equation}
		where 
		\begin{equation}\label{eq: manufactured sol 2}
			\phi(x,t) = \sin(\pi t)\sin(\pi x)\sin(\pi y),
		\end{equation}
		with homogeneous Dirichlet boundary conditions. The chosen parameters for the two test cases are:
		\begin{itemize}
				\item Test case 1 - every parameter is set equal to 1, except for the Poisson ratio $\nu=0.49$.
				\item Test case 2 - we consider parameters corresponding to sandstone saturated with water; see \Cref{tab:poroelastic_constants}. The pressure is scaled by $p_{\!_{\rm ref}}=10^{10}$.
		\end{itemize}
		We choose the polynomial degree $k=1$ and $r=1$ such that the solutions $\bmu_{\t,h},\bmv_{\t,h},\Psi_{\t,h}\in (X_{\t}^{1}(V_{h}^{1}))^{2}$ and $p_{\t,h}\in X_{\t}^{1}(V_{h}^{1})$ are obtained. In addition, we also consider polynomials of degree $k=1$ and $r=2$. The errors and order of convergence in a succession of refined mesh and time step sizes for test case 1 are displayed in \Cref{table:equalorderOne} and \Cref{table:equalorderTwo}. The results confirm our estimate \eqref{eq:mainestimate}. For both $r=1$ and $r=2$, we see the expected order of convergence for $\del(\bmu-\bmu_{\t,h})$. However, when $r=2$, the convergence of the other variables $\bmv_{\t,h},p_{\t,h},\Psi_{\t,h}$ is limited by the temporal discretization.
				\begin{table}
			\centering
			\renewcommand{\arraystretch}{1.1} 
			\begin{tabular}{|c|c|c|c|c|c|c|c|c|}
				\hline
				\!\!$\t,\, h$\!\! & \!\!$\|\del(\bmu-\bmu_{\t,h})\|$\!\! & \!\!EOC\!\! & \!\!$\|\bmv-\bmv_{\t,h}\|$\!\! & \!\!EOC\!\! & \!\!$\|p-p_{\t,h}\|$\!\!  & \!\!EOC\!\! & \!\!$\|\Psi-\Psi_{\t,h}\|$\!\!  & \!\!EOC\!\! \\
				\hline
				\!\!$\t_0/2^{0},\, h_0/2^{0}$\!\! 
				& 3.6860e-01 & 	- & 	3.4710e-01 & 	- & 	2.0841e-02 & 	- & 	8.8025e-02 & 	- \\
				\hline
				\!\!$\t_0/2^{1},\, h_0/2^{1}$\!\! 
				& 1.8903e-01 & 	0.96 & 	1.0959e-01 & 	1.66 & 	5.0024e-03 & 	2.06 & 	4.0902e-02 & 	1.11\\
				\hline
				\!\!$\t_0/2^{2},\, h_0/2^{2}$\!\! 
				& 9.5111e-02 & 	0.99 & 	2.8563e-02 & 	1.94 & 	1.2461e-03 & 	2.01 & 	1.0868e-02 & 	1.91\\
				\hline
				\!\!$\t_0/2^{3},\, h_0/2^{3}$\!\! 
				& 4.7630e-02 & 	0.99 & 	7.1886e-03 & 	1.99 & 	3.1141e-04 & 	2.00 & 	2.7363e-03 & 	1.99\\
				\hline
				\!\!$\t_0/2^{4},\, h_0/2^{4}$\!\! 
				& 2.3824e-02 & 	0.99 & 	1.7992e-03 & 	1.99 & 	7.7850e-05 & 	2.00 & 	6.8465e-04 & 	1.99\\
				\hline
			\end{tabular}
			\caption{Test case 1 - Error and estimated order of convergence at $t=0.1$ with $k=1$ and $r=1$ and in the norm $\|\cdot\|_{L^{2}}$ where $\t_0=0.05, h_0=1/(2\sqrt{2})$.}
			\label{table:equalorderOne}
		\end{table}
		
		In the limit case of vanishing coefficients, when $c_0\to 0, |\calA|\to 0$, a Stokes-type structure is obtained, and therefore, the robustness of the equal-order scheme can not be expected due to the known stability issues of these systems, cf. \Cref{rem:infsup}. Consequently, we also consider inf-sup stable elements $\{\bm{V}^{r+1}_h,V^{r}_h\}$ leading to discrete solutions $(\bmu_{\t,h},\bmv_{\t,h},\Psi_{\t,h})\in (X_{\t}^{1}(V_{h}^{r+1}))^{2}$ and $p_{\t,h}\in X_{\t}^{r}(V_{h}^{1})$ of \eqref{eqs:biotallardADE}. When $r=1$, we see in \Cref{table:error} that we have second-order convergence. This strengthens our conjecture in the introduction that the inf-sup stable approximation converges with order $r+1$ in space.

		\begin{table}
			\centering
			\renewcommand{\arraystretch}{1.1} 
			\begin{tabular}{|c|c|c|c|c|c|c|c|c|}
				\hline
				\!\!$\t,\, h$\!\! & \!\!$\|\del(\bmu-\bmu_{\t,h})\|$\!\! & \!\!EOC\!\! & \!\!$\|\bmv-\bmv_{\t,h}\|$\!\! & \!\!EOC\!\! & \!\!$\|p-p_{\t,h}\|$\!\!  & \!\!EOC\!\! & \!\!$\|\Psi-\Psi_{\t,h}\|$\!\!  & \!\!EOC\!\! \\
				\hline
				\!\!$\t_0/2^{0},\, h_0/2^{0}$\!\! 
				& 6.1169e-02 & 	- & 	1.6741e-02 & 	- & 	1.7487e-02 & 	- & 	1.1743e-02 & 	-\\
				\hline
				\!\!$\t_0/2^{1},\, h_0/2^{1}$ \!\!
				& 1.5081e-02 & 	2.02 & 	2.0368e-03 & 	3.04 & 	3.9418e-03 & 	2.15 & 	4.3313e-03 & 	1.44\\
				\hline
				\!\!$\t_0/2^{2},\, h_0/2^{2}$ \!\!
				& 3.7314e-03 & 	2.01 & 	3.4352e-04 & 	2.57 & 	9.5825e-04 & 	2.04 & 	1.1230e-03 & 	1.95\\
				\hline
				\!\!$\t_0/2^{3},\, h_0/2^{3}$\!\! 
				& 9.2614e-04 & 	2.01 & 	7.4002e-05 & 	2.21 & 	2.3934e-04 & 	2.00 & 	2.7921e-04 & 	2.01\\
				\hline
				\!\!$\t_0/2^{4},\, h_0/2^{4}$\!\! 
				& 2.3091e-04 & 	2.00 & 	1.7712e-05 & 	2.06 & 	5.9872e-05 & 	1.99 & 	6.9623e-05 & 	2.00\\
				\hline
			\end{tabular}
			\caption{Test case 1 - Error and estimated order of convergence at $t=0.1$ with $k=1$ and $r=2$ and in the norm $\|\cdot\|_{L^{2}}$ where $\t_0=0.05, h_0=1/(2\sqrt{2})$.}
			\label{table:equalorderTwo}
		\end{table}
		\begin{table}
			\centering
			\renewcommand{\arraystretch}{1.1} 
			\begin{tabular}{|c|c|c|c|c|c|c|c|c|}
				\hline
				\!\!$\t,\, h$\!\! & \!\!$\|\del(\bmu-\bmu_{\t,h})\|$\!\! & \!\!EOC\!\! & $\|\bmv-\bmv_{\t,h}\|$ & \!\!EOC\!\! & \!\!$\|p-p_{\t,h}\|$\!\!  & \!\!EOC\!\! & \!\!$\|\Psi-\Psi_{\t,h}\|$\!\!  & \!\!EOC\!\! \\
				\hline
				\!\!$\t_0/2^{0}, h_0/2^{0}$\!\! & \!\!6.1166e-02\!\! &$-$& 1.7124e-02
				& $-$ & 7.4390e-02
				& $-$ & 2.7374e-02& $-$ \\
				\hline
				\!\!$\t_0/2^{1}, h_0/2^{1}$\!\! & \!\!1.5078e-02\!\!&$2.02
				$&2.0023e-03 & $3.10$ & 2.2741e-02 & $1.71$ & 9.4740e-03& $1.53$\\
				\hline
				\!\!$\t_0/2^{2}, h_0/2^{2}$\!\! & \!\!3.7310e-03\!\! & $2.01$& 2.9004e-04
				& $2.78
				$ & 5.9040e-03 & $1.95
				$& 2.6946e-03& $1.81$\\
				\hline
				\!\!$\t_0/2^{3}, h_0/2^{3}$\!\! & \!\!9.2613e-04\!\!& $2.01$& 5.4349e-05 & $2.41$ & 1.4846e-03
				& $1.99$ & 6.9783e-04 & $1.95$\\
				\hline
				\!\!$\t_0/2^{4}, h_0/2^{4}$\!\! & \!\!2.3091e-04\!\!& $2.00$&1.2256e-05 & $2.15$ & 3.7084e-04
				& $2.00$ & 1.7656e-04 & $1.98$\\
				\hline
			\end{tabular}
			\caption{Test case 1 - Error and estimated order of convergence at $t=0.1$ for the inf-sup stable pair of Taylor-Hood finite elements spaces $(\bmu_{\t,h},\bmv_{\t,h},p_{\t,h},\Psi_{\t,h})\in\{\bm{V}_h^{2},\bm{V}_h^{2},V^{1}_h,\bm{V}_h^{2}\}$ for the spatial approximation and polynomial degree $k=1$ for the temporal approximation in the norm $\|\cdot\|_{L^{2}}$ where $\t_0=0.05, h_0=1/(2\sqrt{2})$.}
			\label{table:error}
		\end{table}

		For test case 2, when $c_0=1e-10$, in the equal order element case when $r=1$ we see the excepted order of convergence for $\del(\bmu-\bmu_{\t,h})$, see \Cref{table:testcase2equalorderOne}. We also see 2nd-order convergence in the inf-sup stable case, cf. \Cref{table:testcase2InfSup}, inline with the conjecture \eqref{eq:conjecture}. Furthermore, no locking, nonphysical oscillations, or spurious pressure modes were observed in this parameter regime. 
		\begin{table}
			\centering
			\renewcommand{\arraystretch}{1.1} 
			\caption{Test case 2 - Poroelastic parameters corresponding to a sandstone formation saturated with water.}
			\begin{tabular}{c|c | c|c}
				\hline
				\textbf{Parameter} & \textbf{Value} & \textbf{Parameter} & \textbf{Value} \\ \hline
				\(E\) & \(15e9\)  & \(\nu\) & 0.3\\ \hline
				\(\rho_s\) & 2650 &  \(\rho_f\)& 1000  \\ \hline
				\(\phi\) & 0.15 & \(\eta\)  & \(1e-3\) \\ \hline
				\(\alpha\) &  0.7  & \(\eta_k\)  & \(\eta/\rho_f\) \\ \hline
				\(c_0\) & \(1e-10\)  & $F$ & 16.67\\ \hline
				$c_1$ & 2.0953315e-06 & $d_1$ & 1.66666364e-06\\ \hline
			\end{tabular}
			\label{tab:poroelastic_constants}
		\end{table}

		\begin{table}
			\centering
			\renewcommand{\arraystretch}{1.1} 
			\begin{tabular}{|c|c|c|c|c|c|c|c|c|}
				\hline
				\!\!$\t,\, h$\!\! & \!\!$\|\del(\bmu-\bmu_{\t,h})\|$\!\! & \!\!EOC\!\! & \!\!$\|\bmv-\bmv_{\t,h}\|$\!\! & \!\!EOC\!\! & \!\!$\|p-p_{\t,h}\|$\!\!  & \!\!EOC\!\! & \!\!$\|\Psi-\Psi_{\t,h}\|$\!\!  & \!\!EOC\!\! \\
				\hline
				\!\!$\t_0/2^{0},\, h_0/2^{0}$\!\! 
				& 3.7070e-1 & 	- & 	2.4567e-1 & 	- & 	1.8911e-2 & 	- & 	2.5780e-2 & 	-\\
				\hline
				\!\!$\t_0/2^{1},\, h_0/2^{1}$\!\! 
				& 1.8935e-1 & 	0.97 & 	6.32879e-2 & 	1.96 & 	4.9559e-3 & 	1.93 & 	6.6661e-3 & 	1.95\\
				\hline
				\!\!$\t_0/2^{2},\, h_0/2^{2}$\!\! 
				& 9.5150e-2 & 	0.99 & 	1.5976e-2 & 	1.99 & 	1.2558e-3 & 	1.98 & 	1.6806e-3 & 	1.99\\
				\hline
				\!\!$\t_0/2^{3},\, h_0/2^{3}$\!\! 
				& 4.7635e-2 & 	0.99 & 	4.1469e-3 & 	1.94 & 	3.1526e-4 & 	1.99 & 	4.2104e-4 & 	1.99\\
				\hline
			\end{tabular}
			\caption{Test case 2 - Error and estimated order of convergence at $t=0.1$ with $k=1$ and $r=1$ and in the norm $\|\cdot\|_{L^{2}}$ where $\t_0=0.05, h_0=1/(2\sqrt{2})$. }
			\label{table:testcase2equalorderOne}
		\end{table}
		\begin{table}
			\centering
			\renewcommand{\arraystretch}{1.1} 
			\begin{tabular}{|c|c|c|c|c|c|c|c|c|}
				\hline
				\!\!$\t,\, h$\!\! & \!\!$\|\del(\bmu-\bmu_{\t,h})\|$\!\! & \!\!EOC\!\! & \!\!$\|\bmv-\bmv_{\t,h}\|$\!\! & \!\!EOC\!\! & \!\!$\|p-p_{\t,h}\|$\!\!  & \!\!EOC\!\! & \!\!$\|\Psi-\Psi_{\t,h}\|$\!\!  & \!\!EOC\!\! \\
				\hline
				\!\!$\t_0/2^{0},\, h_0/2^{0}$\!\! 
				& 5.7062e-2 & 	- & 	1.6021e-2 & 	- & 	1.9189e-2 & 	- & 	1.6092e-3 & 	-\\
				\hline
				\!\!$\t_0/2^{1},\, h_0/2^{1}$\!\! 
				& 1.4657e-2 & 	1.96 & 	2.0531e-3 & 	2.96 & 	5.0605e-3 & 	1.92 & 	2.4997e-4 & 	2.69\\
				\hline
				\!\!$\t_0/2^{2},\, h_0/2^{2}$\!\! 
				& 3.6916e-3 & 	1.99 & 	2.6456e-4 & 	2.95 & 	1.2851e-3 & 	1.98 & 	4.8051e-05 & 	2.38\\
				\hline
				\!\!$\t_0/2^{3},\, h_0/2^{3}$\!\! 
				& 9.2464e-4 & 	2.00 & 	3.4478e-05 & 	2.93 & 	3.2294e-4 & 	1.99 & 	1.0918e-05 & 	2.14\\
				\hline
			\end{tabular}
			\caption{Test case 2 - Error and estimated order of convergence at $t=0.1$ for the inf-sup stable pair of Taylor-Hood finite elements spaces $(\bmu_{\t,h},\bmv_{\t,h},p_{\t,h},\Psi_{\t,h})\in\{\bm{V}_h^{2},\bm{V}_h^{2},V^{1}_h,\bm{V}_h^{2}\}$ for the spatial approximation and polynomial degree $k=1$ for the temporal approximation in the norm $\|\cdot\|_{L^{2}}$ where $\t_0=0.05, h_0=1/(2\sqrt{2})$.}
			\label{table:testcase2InfSup}
		\end{table}

		\section*{Acknowledgements} JSS and FAR acknowledge the support of the VISTA program, The Norwegian Academy of Science and Letters, and Equinor.

		\appendix
			
			\section{Proof of Lemma \ref{lemma:estimation2} and \ref{lemma:delta1}}\label{appA}
			\begin{proof}[Proof of Lemma \ref{lemma:estimation2}]
				The estimate \eqref{eq:lemma49} follows the ideas of \cite[Lem 3.5]{makridakis2005}, \cite[Lem 4.6]{bause2024} and \cite[Lem 5.6]{kraus2024} with the remark that since $\bmEth^{3}\in (X^{k}_\t(\bm{V}^{r}_h))^{d}\subset C([0,T];\bm{V}^{r}_h)$ we have $\bmEth^{3}(t^{+}_{n-1})=\bmEth^{3}(t_{n-1})$. Note that we have $\bmw_{1}(t_0)=\bmu_0$ and $\bmw_{2}(t_0)=\bmR_h\bmu_1$ from \eqref{eq:special approximation}. Recall the assumption for the initial approximations \eqref{ass:initialapprox}. It then follows that
				\begin{equation}\begin{aligned}
						&\norm{\bmE_{\t,h}^{1}(t_0^{+})}_e^{2}+\frac12\left\|\bm{N}^{\frac12}\begin{pmatrix}
							\bmE^{2}_{\t,h}(t_0^{+})\\\bmEth^{3}(t_0^{+})
						\end{pmatrix}\right\|_{\bm{L}^{2}}^{2}+\frac{c_0}{2}\|e_{\t,h}(t_0^{+})\|_{L^{2}}^{2}	\\
						&\quad\leq c\|\del(\bmR_h\bmu_{0,h}-\bmu_{0,h})\|_{\bm{L}^{2}}^{2}
						+c\left\|\bm{N}^{\frac12}\begin{pmatrix}
							\bmR_h\bmu_{1}-\bmv_{0,h}\\\bmR_h\Psi_{0}-\Psi_{0,h}
						\end{pmatrix}\right\|_{\bm{L}^{2}}^{2}+c\|R_hp_0-p_{0,h}\|_{L^{2}}^{2}
						\leq ch^{2r}.
					\end{aligned}
				\end{equation}
		\end{proof}
		\begin{proof}[Proof of Lemma \ref{lemma:delta1}]
			By the continuity of $\bmu_{\t,h}$, the stability property of $I_{\t}$, the error splitting \eqref{eqs:errorsplit}, and the estimate \eqref{eq:ellipticprojectionerrorVector} and \eqref{eq:special approximation}
			\begin{equation*}
				\begin{aligned}
					\d_{1,n-1}^{+}&=\langle A_h(I_{\t}\bmu\left(t_{n-1}\right)-\bmu\left(t_{n-1}\right)),\bmR_h\bmu\left(t_{n-1}\right)-\bmu_{\t,h}\left(t_{n-1}\right)\rangle\\
					&=\langle A_h(I_{\t}\bmu\left(t_{n-1}\right)-\bmu\left(t_{n-1}\right)),\bmw_1\left(t_{n-1}\right)-\bmu_{\t,h}\left(t_{n-1}\right)\rangle
					\\
					&\quad+\langle A_h(I_{\t}\bmu\left(t_{n-1}\right)-\bmu\left(t_{n-1}\right)),\bmR_h\bmu\left(t_{n-1}\right)-\bmw_1\left(t_{n-1}\right)\rangle\\
					&=\d_{1,n-1}+\langle A_h(I_{\t}\bmu\left(t_{n-1}\right)-\bmu\left(t_{n-1}\right)),\bmR_h\bmu\left(t_{n-1}\right)-\int_{t_{n-2}}^{t_{n-1}}I_\t(\bmR_h\p_t\bmu)dt-\bmR_h\bmu(t_{n-2})\rangle\\
					&=\d_{1,n-1}+\langle A_h(I_{\t}\bmu\left(t_{n-1}\right)-\bmu\left(t_{n-1}\right)),\int_{t_{n-2}}^{t_{n-1}}\bmR_h(\p_t\bmu-I_\t(\p_t\bmu))dt\rangle:= \d_{1,n-1}+r_{n-1}.
				\end{aligned}
			\end{equation*}
			Therefore, we have 
			\begin{equation*}
				\d_{1,n}-\d_{1,n-1}^{+}=\d_{1,n}-\d_{1,n-1}-r_{n-1},
			\end{equation*}
			and obtain
			\begin{equation*}
				\begin{aligned}
					|r_{n-1}|=&\,\left|\langle A_h(I_{\t}\bmu\left(t_{n-1}\right)-\bmu\left(t_{n-1}\right)),\int_{t_{n-2}}^{t_{n-1}}\bmR_h(\p_t\bmu-I_\t(\p_t\bmu))dt\rangle\right|
					\\
					\leq&\, c\t_{n-1}\t_n^{2(k+1)}\left((\varepsilon^{n-1,4}_{t})^{2}+(\varepsilon_t^{n-1,3})^{2}\right).
				\end{aligned}
			\end{equation*}
			For $n=1$, by \eqref{eqs:errorsplit}, \eqref{eq:special approximation} and \eqref{eq:lagrangestability}, with assumption \eqref{ass:initialapprox}
			\begin{equation}
				\begin{aligned}
					\d_{1,0}^{+}&=\langle A_h(I_{\t}\bmu_0-\bmu_0),\bmR_h\bmu_0-\bmu_{0,h}\rangle\leq c\t^{2(k+1)}\|\p_t^{k+1}\bmu_0\|^{2}_{1}+ch^{2(r+1)}\|\bmu_0\|_{r+1}^{2}.
				\end{aligned}
			\end{equation}
			Further, by the inequalities of Cauchy–Schwarz and Young, we obtain
			\begin{equation}\label{eq:delta11}
				\begin{aligned}
					\d_{1,1}&=\langle A_h(I_{\t}\bmu(t_1)-\bmu(t_1)),\bmEth^{1}(t_1)\rangle\leq c\t^{2(k+1)}\|\p_t^{k+1}\bmu(t_1)\|^{2}_{1}+\e\|\del\bmEth^{1}(t_1)\|^{2}.
				\end{aligned}
			\end{equation}
			Then, by the triangle inequality, we get \eqref{eq:deltainitial1}.    
		\end{proof}

		\section{Numerical convergence of ADE formulation towards full convolution integral formulation}\label{appB}
		The ADE formulation \eqref{eqs:biotallardADE} is an approximation of the full convolution formulation \eqref{eqs:BiotAllardconv}. Therefore, we numerically investigate the order of convergence towards a solution of the original formulation written as a first-order system in time when using the ADE approach.   For simplicity, we assume that the dynamic permeability is represented by an exponential decay function in the time-domain ($\mathcal{A}(t)=e^{-t})$ in accordance with a fading memory hypothesis.  All parameters are set to $1$, except the Poisson ratio $\nu=0.49$. We consider the same manufactured solution as before \eqref{eq: manufactured sol 1}-\eqref{eq: manufactured sol 2} on the unit square and $I=(0,0.2]$ homogeneous Dirichlet boundary conditions. However, we note that no forcing right-hand side term is enforced on the auxiliary variable $\Psi$, meaning $\bmF=0$ in \eqref{eqs:biotallardADE} and only constructed source terms for $p$ and $\bmu$ are used. We only  consider the polynomial degrees $k=1$ and $r=1$ and inf-sup stable elements $\{\bm{V}^{r+1}_h,V^{r}_h\}$ for the polynomial degree $k=1$ in time. When $r=1$, see \Cref{table:equalorderADE}, 2nd order convergence is observed for velocity, and the gradient of the displacement exhibits first order convergence. The pressure only shows first-order convergence. The errors and EOC for the inf-sup stable pair are displayed in \Cref{table:infsupADE}. 2nd-order convergence is achieved for the velocity and gradient of the displacement; however, the pressure error appears to degrade upon refinement, while still exhibiting slightly higher than first-order convergence. A possible explanation is the additional approximation error introduced by the ADE formulation of the convolution integral. The error is still smaller than for the equal-order approximation.
		\begin{table}[ht]
			\centering
			\renewcommand{\arraystretch}{1.1} 
			\begin{tabular}{|c|c|c|c|c|c|c|}
				\hline
				\!\!$\t,\, h$\!\! & \!\!$\|\del(\bmu-\bmu_{\t,h})\|$\!\! & \!\!EOC\!\! & \!\!$\|\bmv-\bmv_{\t,h}\|$\!\! & \!\!EOC\!\! & \!\!$\|p-p_{\t,h}\|$\!\!  & \!\!EOC\!\! \\
				\hline
				\!\!$\t_0/2^{0},\, h_0/2^{0}$\!\! 
				& 7.0145e-1 & 	- & 	3.4755e-1 & 	- & 	1.8093e-1& 	- \\
				\hline
				\!\!$\t_0/2^{1},\, h_0/2^{1}$\!\! 
				& 3.5967e-1 & 	0.96 & 	9.6204e-2 & 	1.85 & 	1.0806e-1 & 	0.74 \\
				\hline
				\!\!$\t_0/2^{2},\, h_0/2^{2}$\!\! 
				& 1.8093e-1 & 	0.99 & 	2.4628e-2 & 	1.97 & 	5.9859e-2 & 	0.85  \\
				\hline
				\!\!$\t_0/2^{3},\, h_0/2^{3}$\!\! 
				& 9.0599e-2 & 	0.99 & 	6.1936e-3 & 	1.99 & 	3.1767e-2 & 	0.91 \\
				\hline
				\!\!$\t_0/2^{4},\, h_0/2^{4}$\!\! 
				& 4.5316e-2 & 	0.99 & 	1.5528e-3 & 	1.99 & 	1.6506e-2 & 	0.94\\
				\hline
			\end{tabular}
			\caption{Test case 3 - Error and estimated order of convergence of ADE approximation vs full convolution integral form at $t=0.2$ with $k=1$ and $r=1$ and in the norm $\|\cdot\|_{L^{2}}$ where $\t_0=0.025, h_0=1/(2\sqrt{2})$.}
			\label{table:equalorderADE}
		\end{table}
		\begin{table}[ht]
			\centering
			\renewcommand{\arraystretch}{1.1} 
			\begin{tabular}{|c|c|c|c|c|c|c|}
				\hline
				\!\!$\t,\, h$\!\! & \!\!$\|\del(\bmu-\bmu_{\t,h})\|$\!\! & \!\!EOC\!\! & \!\!$\|\bmv-\bmv_{\t,h}\|$\!\! & \!\!EOC\!\! & \!\!$\|p-p_{\t,h}\|$\!\!  & \!\!EOC\!\! \\
				\hline
				\!\!$\t_0/2^{0},\, h_0/2^{0}$\!\! 
				& 1.1640e-1 & 	- & 	1.5220e-2 & 	- & 	1.0912e-1 & 	- \\
				\hline
				\!\!$\t_0/2^{1},\, h_0/2^{1}$\!\! 
				& 2.8682e-2 & 	2.02 & 	1.7108e-3 & 	3.15 & 	3.3205e-2 & 	1.72 \\
				\hline
				\!\!$\t_0/2^{2},\, h_0/2^{2}$\!\! 
				& 7.0971e-3 & 	2.01 & 	2.3183e-4 & 	2.88 & 	8.9018e-3 & 	1.90  \\
				\hline
				\!\!$\t_0/2^{3},\, h_0/2^{3}$\!\! 
				& 1.7617e-3 & 	2.01 & 	4.067e-05 & 	2.51 & 	2.8886e-3 & 	1.62 \\
				\hline
				\!\!$\t_0/2^{4},\, h_0/2^{4}$\!\! 
				& 4.3924e-4 & 	2.00 & 	1.0091e-05 & 	2.01 & 	1.2793e-3 & 	1.18 \\
				\hline
			\end{tabular}
			\caption{Test case 3 - Error and estimated order of convergence of ADE approximation vs full convolution integral form at $t=0.2$ for the inf-sup stable pair of Taylor-Hood finite elements spaces $(\bmu_{\t,h},\bmv_{\t,h},p_{\t,h},\Psi_{\t,h})\in\{\bm{V}_h^{2},\bm{V}_h^{2}, V^{1}_h,\bm{V}_h^{2}\}$ for the spatial approximation and polynomial degree $k=1$ for the temporal approximation in the norm $\|\cdot\|_{L^{2}}$ where $\t_0=0.025, h_0=1/(2\sqrt{2})$.}
			\label{table:infsupADE}
		\end{table}
		
		Since we see that the ADE formulation converges towards a manufactured solution of the full convolution integral formulation, we now consider how accurate the reconstruction of the dynamic permeability in the frequency domain compares to the JKD dynamic permeability \cite[Eq. 3.4b]{johnson1987}. Under the time harmonic convention $e^{-i\omega t}$, the JKD dynamic permeability for an isotropic medium is expressed as 
			\begin{equation}
				\hat{\calA}(\omega) = \frac{k}{\sqrt{1-\frac{4ia_{\infty}^{2}k^{2}\rho_f\omega}{\Lambda^{2}\eta \phi^{2}}}-\frac{i\alpha_{\infty}k\rho_f\omega}{\eta \phi}   },
			\end{equation}where $k$ is the permeability, $\eta$ is the viscosity, $\phi$ is the porosity, $\rho_f$ is the fluid density, $a_{\infty}$ is the infinite-frequency tortuosity and $\Lambda$ a tunable geometric parameter. In this sign convention, the approximation of the dynamic permeability in the frequency domain is \begin{align}\label{relation:dynamicpermeability}
				\hat{\calA}(\om) \approx \frac{\eta_k }{F}\sum_{i=1}^{N}\frac{d_i}{1-i \om c_i},
			\end{align}
			where $\eta_k=\eta/\rho_f$ and $F=a_{\infty}/\phi$. The constants $c_i,d_i$ are obtained by a numerical procedure from \cite{yvonne2014}. We consider three examples:  a sandstone formation saturated with water where the application of interest is seismic wave propagation in the frequency range 1Hz-1000Hz, a different sandstone formation saturated with water where the application of interest is acoustic borehole logging in the frequency range 250kHz-350kHz, and a constructed sound absorber saturated with air where the application of interest is noise protection. For the last example, we consider a broader frequency range than the audible range. The associated parameters and constants $c_i,d_i$ for each example can be found in \Cref{tab:material_detailed}. In the first example, the comparison is shown in~\Cref{fig:seismic}; the approximation is highly accurate when $N=1$. Similarly, for the second example, the results are displayed in \Cref{fig:borehole}; $N=1$ appears to be sufficient in the frequency range of interest, despite not being as good a fit as in the previous example. In the last example, see \Cref{fig:foam}, $N=1$ does not appear to be sufficient for higher frequencies, and a larger $N$ results in a better approximation. However, we note that at audible frequencies (20Hz-20000Hz), $N=1$ is still accurate. The examples indicate that $N=1$ is a sufficient approximation for applications with a narrow frequency bandwidth or at smaller frequencies. For broader ranges, $N$ needs to be larger to accurately depict the convolution integral. To assess this fully, a comparison with different numerical approaches for the convolution integral form is required and is left as work for the future.
		
		\begin{table}[htbp]
			\centering
			\caption{JKD dynamic permeability parameters for different materials and corresponding coefficients $c_i,d_i$ used in the reconstruction. $N$ represents the number of points used to obtain $c_i,d_i$.}
			\label{tab:material_detailed}
			\begin{tabular}{p{1.2cm} p{3cm} l p{3cm}p{3cm}}
				\toprule
				\textbf{Material}  & \textbf{JKD Parameters}& $N$ & $c_i$ & $d_i$ \\
				\midrule
				Sandstone saturated with water  & $\phi    = 0.15$, $\alpha_{\infty}  = 2.5$,\newline $k    = 1e-13$,   
				$\eta    = 0.001$,    
				\newline $\rho_f  = 1000$ ,  
				$\Lambda = 3.6e-6 $ & 1& $c_1=2.0953315e-06$ &$d_1=1.66666364e-06$ \\
				\midrule
				Sandstone saturated with water  & $\phi    = 0.25$, $\alpha_{\infty}  = 2$, \newline$k    = 1e-12$,    
				$\eta    = 0.001$,    
				\newline$\rho_f  = 1000$ ,  
				$\Lambda = 1e-5 $& 1 & $c_1=5.1360565e-06$&$ d_1=4.69516501e-06$ \\
				\midrule
				\multirow{3}{1.2cm}{Sound absorber saturated with air}  & \multirow{3}{3cm}{$\phi    = 0.35$, $\alpha_{\infty}  = 2$, $k=  5e-11$,    
					$\eta    = 1.84e-5$    
					$\rho_f  = 1.2$,  
					$\Lambda = 40e-6$} & 1 & $c_1=2.52370572e-05$ &$ d_1=1.86035942e-05$ \\\cmidrule{3-5}
				&  & 2& $c_1=2.44254268e-06$ $c_2= 2.58926623e-05$&$ d_1=4.84063378e-07$ $d_2= 1.81470367e-05$ \\\cmidrule{3-5}
				& & 4& $c_1 =6.05114220e-07$ $c_2 = 2.46471586e-06$ $c_3 = 7.25661720e-06$ $c_4 = 2.63886451e-05$ & $d_1 = 5.26203916e-08$ $d_2 = 2.37157395e-07$  $d_3 = 7.04233051e-07$ $d_4 = 1.76394927e-05$ \\
				\bottomrule
			\end{tabular}
		\end{table}
		\begin{figure}[ht]
			\centering
			\includegraphics[width=0.85\linewidth]{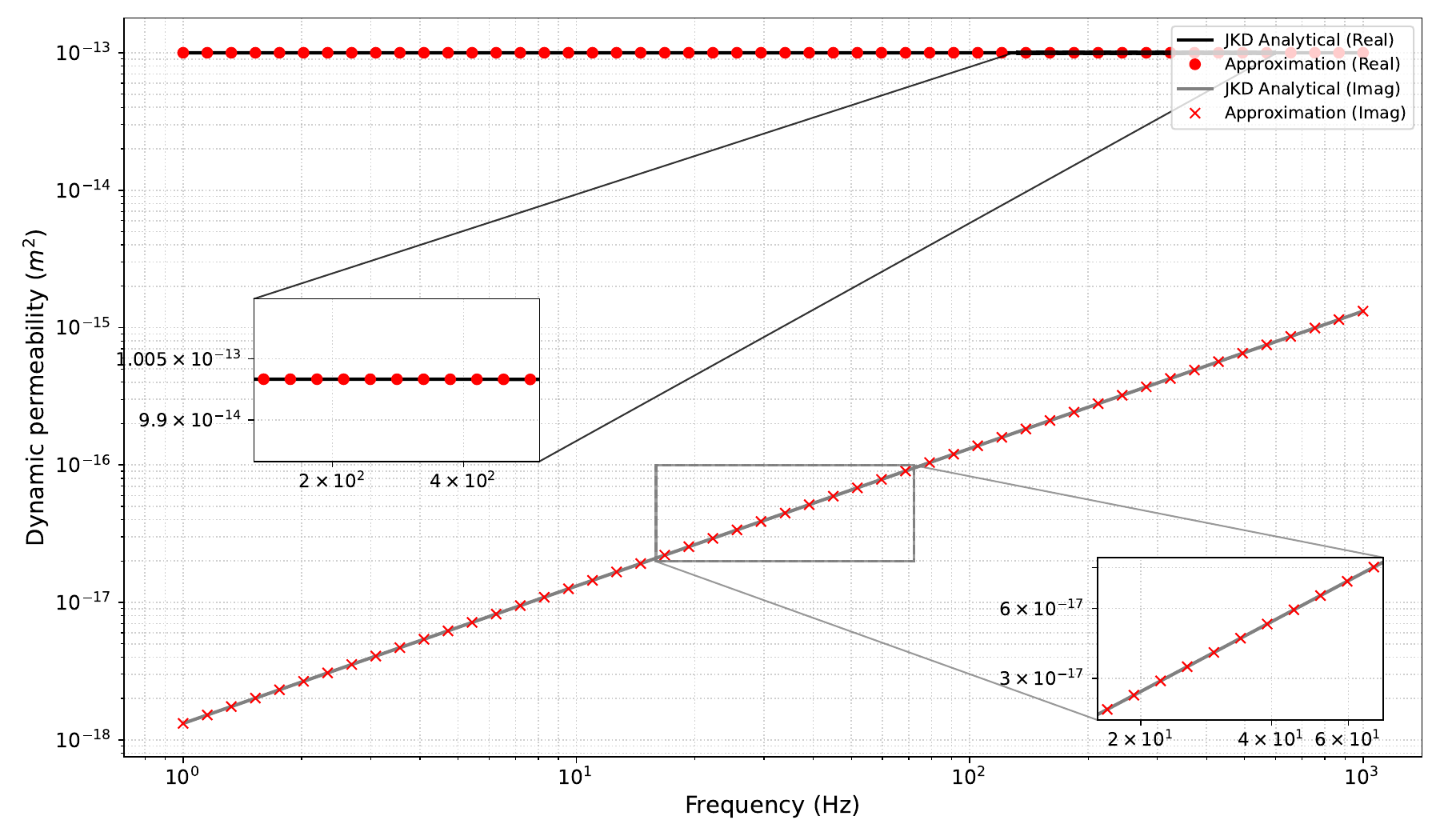}
			\caption{Comparison of JKD dynamic permeability vs approximation in frequency domain - Sandstone formation saturated with water in the frequencies 1Hz-1000Hz relevant for induced seismicity in the subsurface.}
			\label{fig:seismic}
		\end{figure}

		\begin{figure}[ht]
			\centering
			\includegraphics[width=0.85\linewidth]{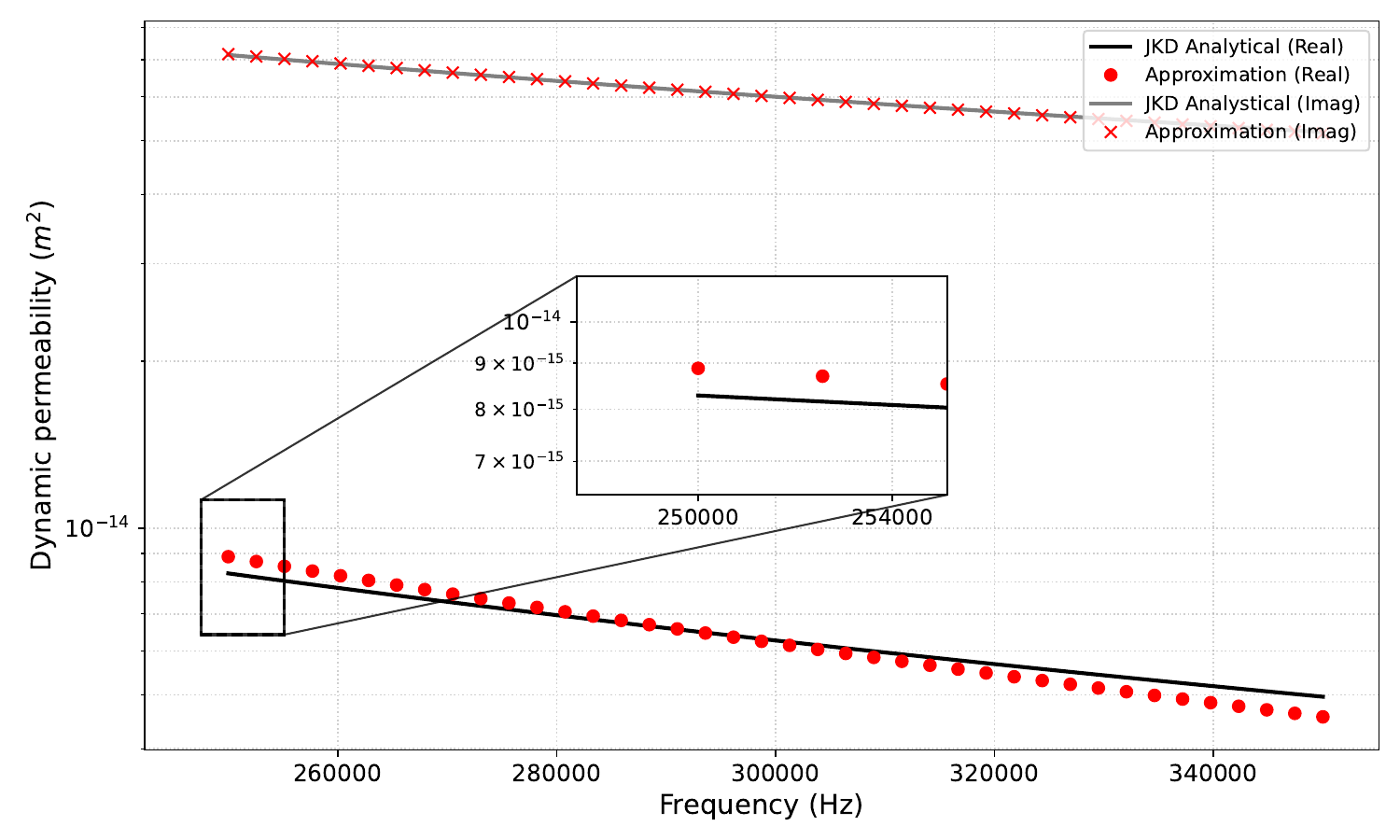}
			\caption{Comparison of JKD dynamic permeability vs approximation in frequency domain - Acoustic borehole logging in a sandstone formation saturated with water in the relevant frequencies 250kHz-350kHz.}
			\label{fig:borehole}
		\end{figure}
		\begin{figure}[ht]
			\centering
			\includegraphics[width=0.85\linewidth]{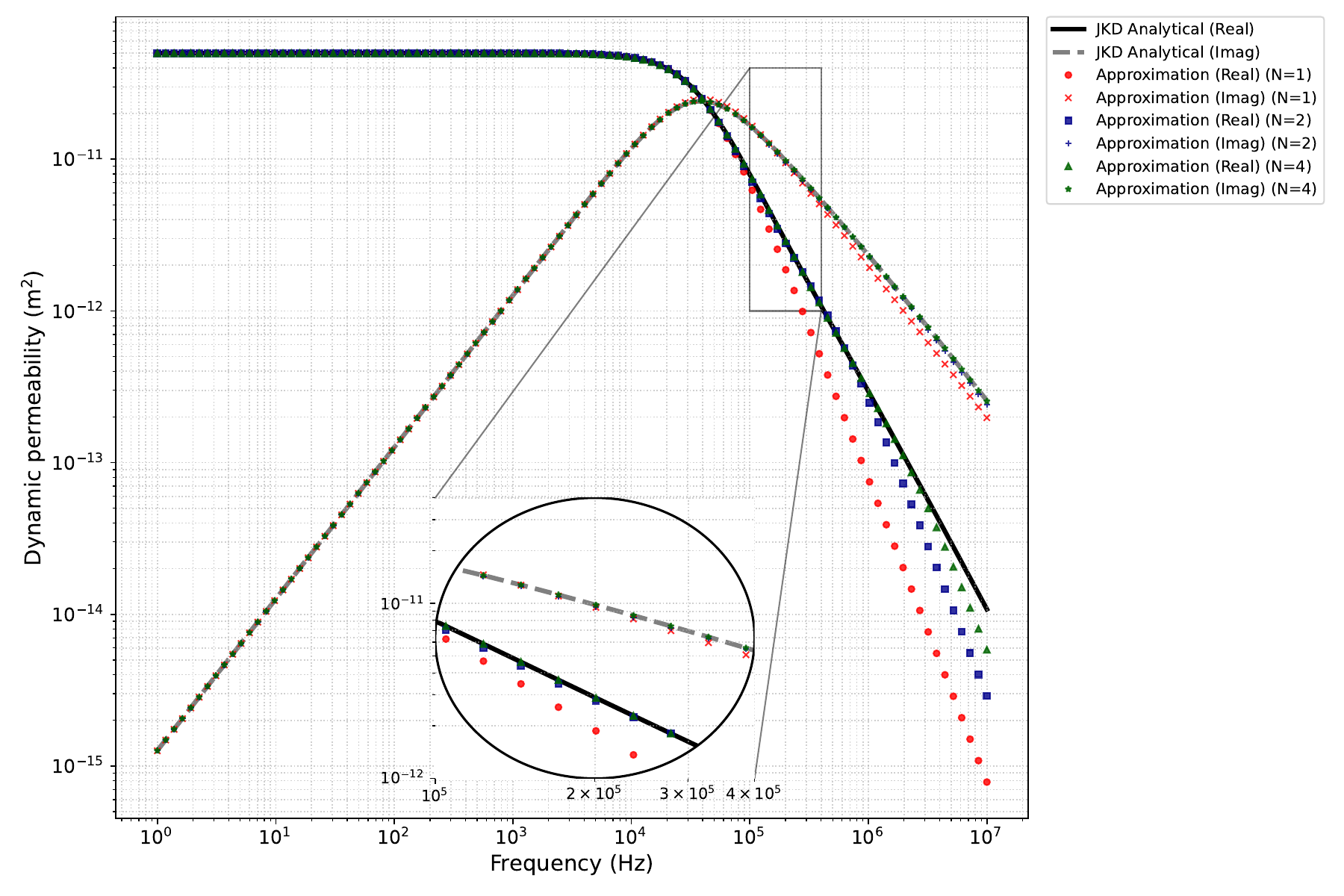}
			\caption{Comparison of JKD dynamic permeability vs approximation in frequency domain - Sound absorber saturated with air for noise protection over a broad frequency range.}
			\label{fig:foam}
		\end{figure}


	\bibliography{sn-bibliography}
	
\end{document}